\documentclass[11pt,a4paper]{amsart}
\usepackage[utf8]{inputenc}
\usepackage[english]{babel}
\usepackage{amsmath}
\usepackage{amsfonts}
\usepackage{amssymb}
\usepackage{mathabx}
\usepackage{graphicx}
\usepackage{xcolor}
\usepackage{float}
\usepackage{fullpage}
\usepackage{hyperref}

\newcommand{\R}{\mathbb R}
\newcommand{\E}{\mathbb E}
\newcommand{\Z}{\mathbb Z}
\newcommand{\N}{\mathbb N}
\newcommand{\J}{\mathbb J}

\newcommand{\1}{\mathbf 1}

\renewcommand{\O}{\mathcal O}

\newtheorem{thm}{Theorem}[section]
\newtheorem{lemma}[thm]{Lemma}
\newtheorem{defn}[thm]{Definition}
\newtheorem{prop}[thm]{Proposition}
\newtheorem{cor}[thm]{Corollary}

\theoremstyle{remark}
\newtheorem{rem}[thm]{Remark}

\numberwithin{equation}{section}

\makeatletter
\def\widebreve{\mathpalette\wide@breve}
\def\wide@breve#1#2{\sbox\z@{$#1#2$}%
     \mathop{\vbox{\m@th\ialign{##\crcr
\kern0.08em\brevefill#1{0.6\wd\z@}\crcr\noalign{\nointerlineskip}%
                    $\hss#1#2\hss$\crcr}}}\limits}
\def\brevefill#1#2{$\m@th\sbox\tw@{$#1($}%
  \hss\resizebox{#2}{\wd\tw@}{\rotatebox[origin=c]{90}{\upshape(}}\hss$}
\makeatletter

\author{Titus Lupu \and Christophe Sabot \and Pierre Tarrès}
\address {CNRS and LPSM, UMR 8001,
Sorbonne Université,
4 place Jussieu,
75252 Paris cedex 05,
France}
\email
{titus.lupu@upmc.fr}

\address {
Institut Camille Jordan,
Université Lyon 1, 
43 bd. du 11 nov. 1918,
69622 Villeurbanne cedex,
France}
\email
{sabot@math.univ-lyon1.fr}

\address {NYU-ECNU Institute of Mathematical Sciences at NYU Shanghai, China; Courant Institute of Mathematical Sciences, New York, USA; CNRS and Universit\'e Paris-Dauphine, PSL Research University, Ceremade, Paris, France}
\email
{tarres@nyu.edu}

\title{Inverting the Ray-Knight identity on the line}

\begin{document}

\begin{abstract}
Using a divergent Bass-Burdzy flow we construct a self-repelling 
one-dimensional diffusion.
Heuristically, it can be interpreted as a solution to an SDE with a singular drift involving a derivative of the local time.
We show that this self-repelling diffusion inverts the second Ray-Knight identity on the line. 
The proof goes through an approximation by
a self-repelling jump processes that has been previously shown by the authors to invert the Ray-Knight identity in the discrete.
\end{abstract}


\maketitle

\section{Introduction and presentation of results}
\label{SecIntro}

\subsection*{Ray-Knight identity on $\R$}
We will construct a continuous self-repelling one-dimensional diffusion, involved in the inversion of the Ray-Knight identity on $\R$. We start by recalling the latter.

Given $a\geq 0$, $(\phi^{(a)}(x))_{x\in\R}$ will denote a massless Gaussian free field on $\R$
conditioned to be $a$ at $x=0$, 
that is to say 
$(\phi^{(a)}(x)/\sqrt{2})_{x\geq 0}$ and 
$(\phi^{(a)}(-x)/\sqrt{2})_{x\geq 0}$ are two independent standard Brownian motions starting from $a/\sqrt{2}$.

\begin{thm}[Ray-Knight 
\cite{Ray1963Sojourn,Knight1963Sojourn,EKMRS2000RK,
RevuzYor1999BMGrundlehren,MarcusRosen2006MarkovGaussianLocTime,
Sznitman2012LectureIso}]
\label{ThmRK}
Fix $a>0$. Let $(\beta_{t})_{t\geq 0}$ be a standard Brownian motion starting from $0$ and 
let $\ell^{\beta}_{t}(x)$ be its local time process. 
Let
$\tau_{a^{2}/2}^{\beta}$ be the stopping time
\begin{displaymath}
\tau^{\beta}_{a^{2}/2}=\inf\lbrace t\geq 0\vert\ell^{\beta}_{t}(0)
>a^{2}/2\rbrace.
\end{displaymath}
Let $(\phi^{(0)}(x))_{x\in\R}$ be a massless Gaussian free field on $\R$ conditioned to be $0$ at $x=0$,
independent from the Brownian motion $\beta$.
Then the field
\begin{displaymath}
(\phi^{(0)}(x)^{2}/2+\ell^{\beta}_{\tau^{\beta}_{a^{2}/2}}(x))_{x\in\R}
\end{displaymath}
has the same law as the field
$(\phi^{(a)}(x)^{2}/2)_{x\in\R}$.
\end{thm}

The original formulation of Ray \cite{Ray1963Sojourn} and
Knight \cite{Knight1963Sojourn} is different. It states that
$(\ell^{\beta}_{\tau^{\beta}_{a^{2}/2}}(x))_{x\geq 0}$ is a squared Bessel process of dimension $0$, 
starting from $a^{2}/2$ at $x=0$
(see also \cite{RevuzYor1999BMGrundlehren}, Section XI.2).
$(\phi^{(0)}(x)^{2}/2)_{x\geq 0}$ is by definition a squared Bessel process of dimension $1$, and by additivity property of squared Bessel processes, 
$(\phi^{(0)}(x)^{2}/2+\ell^{\beta}_{\tau^{\beta}_{a^{2}/2}}(x))
_{x\geq 0}$
is a squared Bessel process of dimension $1=1+0$, 
starting from $a^{2}/2=0+a^{2}/2$, 
the same as $(\phi^{(a)}(x)^{2}/2)_{x\geq 0}$.
In Theorem \ref{ThmRK} we use a reformulation of the Ray-Knights identity that generalizes to a much wider setting, such as any discrete electrical network, 
and continuum setting in dimension 2 and 3 after a Wick renormalization of the square of the GFF 
\cite{EKMRS2000RK,MarcusRosen2006MarkovGaussianLocTime,
Sznitman2012LectureIso}. It also makes the connection to 
Brydges-Fröhlich-Spencer-Dynkin's isomorphism 
\cite{BFS82Loop,Dynkin1984Isomorphism, 
Dynkin1984IsomorphismPresentation} and 
Symanzik's identities in Euclidean Quantum Field Theory
\cite{Symanzik65Scalar,Symanzik66Scalar,Symanzik1969QFT}.

Theorem \ref{ThmRK} provides a way to couple on the
same probability space the triplet
$(\phi^{(0)},\beta,\phi^{(a)})$.
We formalize this in the following definition.

\begin{defn}
\label{Def RK coupl}
Fix $a>0$.
We say that the triplet $(\phi^{(0)},\beta,\phi^{(a)})$
satisfies a Ray-Knight coupling if
the following conditions are satisfied.
\begin{itemize}
\item The process $(\phi^{(0)}(x))_{x\in\R}$
is distributed like a massless Gaussian free field on $\R$
conditioned to be $0$ at $x=0$.
\item The process $(\beta_{t})_{t\geq 0}$
is a standard Brownian motion on $\R$ starting from $0$.
\item The processes $\phi^{(0)}$ and $\beta$ are independent.
\item For every $x\in \R$,
\begin{displaymath}
\phi^{(a)}(x)^{2} = \phi^{(0)}(x)^{2} 
+ 2 \ell^{\beta}_{\tau^{\beta}_{a^{2}/2}}(x).
\end{displaymath}
\item For every $x\in\R$ such that
$(\phi^{(a)})^{2}$ is strictly positive
on $[0,x]$, respectively $[x,0]$,
one has $\phi^{(a)}(x)>0$.
For all other $x\in\R$, 
$\phi^{(a)}(x)= \phi^{(0)}(x)$.
\end{itemize}
\end{defn}

It follows from Theorem \ref{ThmRK} that
$\phi^{(a)}$ in a Ray-Knight coupling is distributed like
a massless Gaussian free field on $\R$
conditioned to be $a$ at $x=0$.

\subsection*{Inversion of the Ray-Knight identity}
Given a Ray-Knight coupling of $(\phi^{(0)},\beta,\phi^{(a)})$,
we are interested in the conditional law of
the stochastic process
$\big(\beta_{\tau^{\beta}_{a^{2}/2}-t}\big)_{t}$
knowing $\phi^{(a)}$.

The Ray-Knight identity of Theorem \ref{ThmRK} generalizes to discrete electrical networks and symmetric Markov jump processes on them
\cite{EKMRS2000RK,MarcusRosen2006MarkovGaussianLocTime,
Sznitman2012LectureIso}. This is known as second generalized Ray-Knight identity. The inversion in the discrete setting was done in
\cite{ST2016InvRK,LST2017InvRK}. 
This inversion involves a nearest neighbor self-repelling jump process on the network. 
More precisely,
the jump rate at time $t$ from a vertex $x_{1}$ to a neighbor $x_{2}$ is given by
\begin{equation}
\label{EqIntRate}
C(x_{1},x_{2})\dfrac{(\Phi(x_{2})^{2}-2L_{t}(x_{2}))^{\frac{1}{2}}}
{(\Phi(x_{1})^{2}-2L_{t}(x_{1}))^{\frac{1}{2}}},
\end{equation}
where $C(x_{1},x_{2})$ is a fixed conductance,
$L_{t}(x_{j})$ is the time spent at $x_{j}$ by the jump process before time $t$, and $\Phi$ is a field on the vertices, considered as an initial condition. In the inversion of Ray-Knight $\Phi$ is random, distributed as a discrete Gaussian free field. In Section \ref{SecDiscr} we detail this in the setting of a discrete subset of $\R$.
Also note that this self-repelling jump process
with jump rates \eqref{EqIntRate} is up to a time change
the vertex-diminished jump process (VDJP)
studied in \cite{ST2016InvRK,BauerschmidtHelmuthSwan2}.

If one takes a one-dimensional fine mesh lattice and renormalizes the jump rates \eqref{EqIntRate}, then on a purely formal level, without dealing with the convergence or the meaning of the terms involved, one gets the following equation for a continuous self-repelling diffusion:
\begin{equation}
\label{MainEq}
d\widecheck{X}_{t}=
dW_{t}+
``\dfrac{1}{2}\partial_{x}\log(\check{\lambda}_{t}(x))\Big
\vert_{x=\widecheck{X}_{t}} dt",
\qquad
\check{\lambda}_{t}(x)=\check{\lambda}_{0}(x)-
2\check{\ell}_{t}(x).
\end{equation}
There $\widecheck{X}_{t}$ is a continuous stochastic process on an interval $I$, $\check{\lambda}_{0}$ a continuous function from $I$ to
$(0,+\infty)$, $W_{t}$ is a standard Brownian motion and 
$\check{\ell}_{t}(x)$ is the local time process of $\widecheck{X}_{t}$. 
We will call $\check{\lambda}_{t}$ the
\textit{occupation profile at time} $t$. 
Our process $\widecheck{X}_{t}$ is defined up to a finite time
\begin{displaymath}
\widecheck{T}=\sup\lbrace t\geq 0\vert 
\check{\lambda}_{t}(\widecheck{X}_{t})>0\rbrace.
\end{displaymath}
We will also assume that
\begin{equation}
\label{EqCond}
\int_{\inf I} \check{\lambda}_{0}(x)^{-1} dx = +\infty,
\qquad
\int^{\sup I} \check{\lambda}_{0}(x)^{-1} dx = +\infty
\end{equation}
and say that $\check{\lambda}_{0}$ is \textit{admissible}.
This is a condition for not reaching the boundary of $I$ in finite time.
$\widecheck{X}_{t}$ is a self-repelling process that tends to avoid places it has visited a lot, yet we will see that a.s. it will eventually exhaust the occupation profile at some location in finite time 
$\widecheck{T}$.
As we will further see, this self-repelling process appears in the inversion of the Ray-Knight identity in the continuous one-dimensional setting.

The equation \eqref{MainEq} is not a classical SDE.
It is not immediately clear how to make sense of the drift term
$\dfrac{1}{2}\partial_{x}\log(\check{\lambda}_{t}(x))\Big
\vert_{x=\widecheck{X}_{t}} dt$, as $x\mapsto\check{\ell}_{t}(x)$ will not be differentiable for $t>0$, and moreover there will not be a change of scale under which it will be differentiable for all $t>0$. 
So the problem is not only to solve \eqref{MainEq} by an approximation scheme, the problem is already to give an appropriate meaning to being a solution to \eqref{MainEq}. The equation \eqref{MainEq} is also somewhat misleading, as we believe that a solution $\widecheck{X}_{t}$ would not be a semi-martingale, admitting an adapted decomposition into a Brownian motion plus a drift term with zero quadratic variation, but with an infinite total variation. 
See \cite{HuWarren00BBFlow,LST2018LRM} for a discussion on this point.

However, it turns out that the equation
\eqref{MainEq} is in some sense exactly solvable, and in this paper we will give the explicit solution 
which involves a divergent bifurcating stochastic
flow of diffeomorphisms of $\R$ introduced by Bass and Burdzy in
\cite{BassBurdzy99StochBiff}. Our construction here is similar to that of 
\cite{LST2018LRM}, where we introduced a reinforced diffusion constructed out of a different, convergent, Bass-Burdzy flow.

\subsection*{Heuristic reduction to a Bass-Burdzy flow}

Next we explain a non-rigorous heuristic derivation of an explicit solution to
\eqref{MainEq}. A similar heuristic appears in the introduction to
\cite{LST2018LRM}.

Assume that for $t_{0}>0$, $\bar{X}^{(t_{0})}_{t}$ is a continuous process coinciding with $\widecheck{X}_{t}$ on 
$[0,t_{0}]$, and after time $t_{0}$ continues as a Markovian diffusion
with infinitesimal generator
\begin{displaymath}
\dfrac{1}{2}\dfrac{d^{2}}{dx^{2}}+
\dfrac{1}{2}\partial_{x}\check{\lambda}_{t_{0}}
\dfrac{d}{dx}.
\end{displaymath}
In other words, there is no additional self-repulsion after time
$t_{0}$.
Then after time $t_{0}$, $\bar{X}^{(t_{0})}_{t}$ is a scale and time changed Brownian motion. 
Given $\bar{S}_{t_{0}}$ an anti-derivative of
$\check{\lambda}_{t_{0}}^{-1}$,
$(\bar{S}_{t_{0}}(\bar{X}^{(t_{0})}_{t}))_{t\geq t_{0}}$ is a local martingale. By further performing the time change
\begin{displaymath}
du=\check{\lambda}_{t_{0}}(\bar{X}^{(t_{0})}_{t})^{-2} dt
\end{displaymath}
we get a standard Brownian motion.

Then it is reasonable to assume that near time $t_{0}$, 
$\widecheck{X}_{t}$ is close to $\bar{X}^{(t_{0})}_{t}$. The idea is to let the change of scale depend on time. 
Assume there is a flow of changes of scales 
$\widecheck{S}_{t}: I\rightarrow\R$, 
such that $\widecheck{S}_{t}$ is an anti-derivative 
of $\check{\lambda}_{t}^{-1}$,
and such that $\widecheck{S}_{t}(\widecheck{X}_{t})$ is a local martingale. 
Consider $u(t)$ the time change given by
\begin{displaymath}
du=\check{\lambda}_{t}(\widecheck{X}_{t})^{-2} dt,
\end{displaymath}
and $t(u)$ the inverse time change.
Assume that,
by analogy with the Markovian case, 
$\widecheck{S}_{t(u)}(\widecheck{X}_{t(u)})_{u\geq 0}$ is a standard Brownian motion $(B_{u})_{u\geq 0}$. 
Let $x_{1}<x_{2}\in I$. Then
\begin{eqnarray}
\nonumber
\dfrac{d}{du}(\widecheck{S}_{t(u)}(x_{2})
-\widecheck{S}_{t(u)}(x_{1}))
&=&\dfrac{dt}{du}
\dfrac{d}{dt}\int_{x_{1}}^{x_{2}}
\check{\lambda}_{t}(x)^{-1} dx\\
\nonumber
&=&\check{\lambda}_{t}(\widecheck{X}_{t})^{2}
\dfrac{d}{dt}
\int_{0}^{t}\1_{x_{1}<\widecheck{X}_{s}<x_{2}}
2\check{\lambda}_{s}(\widecheck{X}_{s})^{-2}ds\\
\label{EqHeuristic}
&=& 2 \check{\lambda}_{t}(\widecheck{X}_{t})^{2}
\check{\lambda}_{t}(\widecheck{X}_{t})^{-2}
\1_{x_{1}<\widecheck{X}_{t}<x_{2}}\\
\nonumber
&=&
2\1_{x_{1}<\widecheck{X}_{t}<x_{2}}\\
\nonumber
&=&
2\1_{\widecheck{S}_{t(u)}(x_{1})<B_{u}<
\widecheck{S}_{t(u)}(x_{2})}.
\end{eqnarray}
This implies that
$\frac{d}{du}\widecheck{S}_{t(u)}(x)$
is of form
\begin{displaymath}
\dfrac{d\widecheck{S}_{t(u)}(x)}{du}=
\1_{\widecheck{S}_{t(u)}(x)>B_{u}}-
\1_{\widecheck{S}_{t(u)}(x)<B_{u}}
+ f(u),
\end{displaymath}
for some function $f(u)$ not depending on $x\in I$.
Further, it is reasonable to assume that the left and the right sides of
$\widecheck{X}_{t}$ play symmetric roles,
and thus $f(u)\equiv 0$.
Then, we get that
\begin{displaymath}
\forall x\in I, \dfrac{d\widecheck{S}_{t(u)}(x)}{du}=
\1_{\widecheck{S}_{t(u)}(x)>B_{u}}-
\1_{\widecheck{S}_{t(u)}(x)<B_{u}}.
\end{displaymath}
This is an equation studied by Bass and Burdzy in
\cite{BassBurdzy99StochBiff}. In the sequel we will construct
$\widecheck{X}_{t}$ out of the flow of solutions to the above equation.

Note that if in the equation \eqref{MainEq}, one replaced the
$\dfrac{1}{2}$ in front of 
$\dfrac{1}{2}\partial_{x}\log(\check{\lambda}_{t}(x))\Big
\vert_{x=\widecheck{X}_{t}} dt$ by a different positive constant, one would not get an as simple explicit solution. Indeed, the cancellation of powers of $\check{\lambda}_{t}(\widecheck{X}_{t})$ as in
\eqref{EqHeuristic} would not occur.

\subsection*{Construction of a self-repelling diffusion out of a divergent Bass-Burdzy flow}

The divergent Bass-Burdzy flow is given by the differential equation
\begin{equation}
\label{EqBB}
\dfrac{dY_{u}}{du}=
\left\lbrace
\begin{array}{ll}
1 & \text{if}~Y_{u}>B_{u}, \\ 
-1 & \text{if}~Y_{u}<B_{u},
\end{array}
\right. 
\end{equation}
where $B_{u}$ is a standard Brownian motion starting from $0$. 
The behavior at times when $Y_{u}=B_{u}$ is not specified. It is shown in \cite{BassBurdzy99StochBiff} that given an initial condition, there is a.s. a unique solution defined for all positive times that is Lipschitz continuous. Moreover, these Lipschitz continuous solutions form a flow of  increasing $\mathcal{C}^{1}$ diffeomorphisms of $\R$, 
$(\widecheck{\Psi}_{u})_{u\geq 0}$. For the properties of this flow, we refer to
\cite{BassBurdzy99StochBiff,HuWarren00BBFlow,Attanasio10DiscDrift}.

Define
\begin{displaymath}
\check{\xi}_{u}=(\widecheck{\Psi}_{u})^{-1}(B_{u}).
\end{displaymath}
$(\widecheck{\Psi}_{u})_{u\geq 0}$ satisfies a bifurcation property
\cite{BassBurdzy99StochBiff}: there is a finite random value
$y_{\rm bif}\in\R$, such that for $y>y_{\rm bif}$, 
$\widecheck{\Psi}_{u}(y)>B_{u}$ for $u$ large enough, and
$\lim_{+\infty}\widecheck{\Psi}_{u}(y)=+\infty$,
for $y<y_{\rm bif}$, $\widecheck{\Psi}_{u}(y)<B_{u}$ 
for $u$ large enough and
$\lim_{+\infty}\widecheck{\Psi}_{u}(y)=-\infty$, and
$\lbrace u\geq 0\vert\widecheck{\Psi}_{u}(y_{\rm bif})=B_{u}\rbrace$
is unbounded. Moreover,
\begin{displaymath}
y_{\rm bif} = \lim_{u\to +\infty}\check{\xi}_{u}.
\end{displaymath}
The process $(\check{\xi}_{u})_{u\geq 0}$ admits
\cite{BassBurdzy99StochBiff,HuWarren00BBFlow} a family of local times
$\widecheck{\Lambda}_{u}(y)$ continuous in $(y,u)$, such that for any $f$ bounded Borel measurable function on $\R$ and $u\geq 0$,
\begin{displaymath}
\int_{0}^{u} f(\check{\xi}_{v}) dv=
\int_{\R} f(y) \widecheck{\Lambda}_{u}(y) dy.
\end{displaymath}
Moreover, these local times are related to the spatial derivative of the flow as follows:
\begin{displaymath}
\dfrac{\partial}{\partial y}\widecheck{\Psi}_{u}(y)=
1+2\widecheck{\Lambda}_{u}(y).
\end{displaymath}
For all $u\geq 0$, $\widecheck{\Lambda}_{u}(y)$,
$\dfrac{\partial}{\partial y}\widecheck{\Psi}_{u}(y)$ and
$\dfrac{\partial}{\partial y}(\widecheck{\Psi}_{u})^{-1}(y)$ are 
locally $1/2-\varepsilon$ Hölder continuous in $y$.

Next we give the construction of $\widecheck{X}_{t}$ out of the flow
$(\widecheck{\Psi}_{u})_{u\geq 0}$.

\begin{defn}
\label{DefMain}
Let $x_{0}\in I$.
Let be the change of scale
\begin{displaymath}
\widecheck{S}_{0}(x)=\int_{x_{0}}^{x}\check{\lambda}_{0}(r)^{-1} dr, ~~x\in I,
\end{displaymath}
and $\widecheck{S}_{0}^{-1}$ the inverse change of scale. Consider the change of time $t(u)$ from $u$ to $t$ (and
$u(t)$ the inverse time change) given by
\begin{equation}
\label{EqChgTime}
dt = \check{\lambda}_{0}(\widecheck{S}_{0}^{-1}(\check{\xi}_{u}))^{2}
(1+2\widecheck{\Lambda}_{u}(\check{\xi}_{u}))^{-2}du.
\end{equation}
Let
\begin{displaymath}
\widecheck{T}=\int_{0}^{+\infty}
\check{\lambda}_{0}(\widecheck{S}_{0}^{-1}(\check{\xi}_{u}))^{2}
(1+2\widecheck{\Lambda}_{u}(\check{\xi}_{u}))^{-2}du.
\end{displaymath}
Set 
$\widecheck{X}_{t}=\widecheck{S}_{0}^{-1}(\check{\xi}_{u(t)})$,
for $t\in[0,\widecheck{T})$.
\end{defn}

We will call $(B_{u})_{u\geq 0}$ the \textit{driving Brownian motion} of $\widecheck{X}_{t}$.

Note that 
\begin{eqnarray*}
\widecheck{T}&=&
\int_{0}^{+\infty}
\check{\lambda}_{0}(\widecheck{S}_{0}^{-1}(\check{\xi}_{u}))^{2}
(1+2\widecheck{\Lambda}_{u}(\check{\xi}_{u}))^{-2}du
=\int_{\R}\int_{0}^{+\infty}
\check{\lambda}_{0}(\widecheck{S}_{0}^{-1}(y))^{2}
(1+2\widecheck{\Lambda}_{u}(y))^{-2} d_{u}\widecheck{\Lambda}_{u}(y) dy\\
&=&
\dfrac{1}{2}\int_{\R}\check{\lambda}_{0}(\widecheck{S}_{0}^{-1}(y))^{2}
(1-(1+2\widecheck{\Lambda}_{+\infty}(y))^{-1})dy
\\
&\leq&\dfrac{1}{2}(\sup_{u\geq 0}\check{\xi}_{u}-
\inf_{u\geq 0}\check{\xi}_{u})
\sup_{[\widecheck{S}_{0}^{-1}(\inf_{u\geq 0}\check{\xi}_{u}),
\widecheck{S}_{0}^{-1}(\sup_{u\geq 0}\check{\xi}_{u})]}
\check{\lambda}_{0}^{2}.
\end{eqnarray*}
Since $\check{\xi}_{u}$ converges at $+\infty$ and thus has a bounded range, $\widecheck{T}< +\infty$ a.s. The process $\widecheck{X}_{t}$ has local times
\begin{displaymath}
\check{\ell}_{t}(x)=
\check{\lambda}_{0}(x)
(1-(1+2\widecheck{\Lambda}_{u(t)}(\widecheck{S}_{0}(x)))^{-1}).
\end{displaymath}
Indeed, for $f$ a measurable bounded function on $I$,
\begin{eqnarray*}
\int_{0}^{t_{1}}f(\widecheck{X}_{t}) dt&=&
\int_{0}^{t_{1}}f(\widecheck{S}_{0}^{-1}(\check{\xi}_{u(t)}))dt=
\int_{0}^{u(t_{1})}f(\widecheck{S}_{0}^{-1}(\check{\xi}_{u}))
\check{\lambda}_{0}(\widecheck{S}_{0}^{-1}(\check{\xi}_{u}))^{2}
(1+2\widecheck{\Lambda}_{u}(\check{\xi}_{u}))^{-2}
du
\\
&=&\int_{\R}\int_{0}^{u(t_{1})}
f(\widecheck{S}_{0}^{-1}(y))
\check{\lambda}_{0}(\widecheck{S}_{0}^{-1}(y))^{2}
(1+2\widecheck{\Lambda}_{u}(y))^{-2}
d_{u}\widecheck{\Lambda}_{u}(y) dy\\
&=&\dfrac{1}{2}\int_{\R}
f(\widecheck{S}_{0}^{-1}(y))
\check{\lambda}_{0}(\widecheck{S}_{0}^{-1}(y))^{2}
(1-(1+2\widecheck{\Lambda}_{u(t_{1})}(y))^{-1}) dy\\
&=&\dfrac{1}{2}\int_{I}
f(x)\check{\lambda}_{0}(x)
(1-(1+2\widecheck{\Lambda}_{u(t_{1})}(\widecheck{S}_{0}(x)))^{-1}) dx.
\end{eqnarray*}
Set 
\begin{displaymath}
\check{\lambda}_{t}(x)=\check{\lambda}_{0}-2\check{\ell}_{t}(x)=
\check{\lambda}_{0}(x)(1+2\widecheck{\Lambda}_{u(t)}(\widecheck{S}_{0}(x)))^{-1}.
\end{displaymath}
\textit{A posteriori}, the change of time \eqref{EqChgTime} is
\begin{displaymath}
dt=\check{\lambda}_{t}(\widecheck{X}_{t})^{2} du.
\end{displaymath}
We see that for all $t\in [0,\widecheck{T})$ and
$x\in I$, $\check{\lambda}_{t}(x)>0$. 
Note that $\widecheck{X}_{\widecheck{T}}=(\widecheck{S}_{0})^{-1}(y_{\rm bif})$.
Moreover,
\begin{displaymath}
\lim_{t\to \widecheck{T}}
\check{\lambda}_{t}(\widecheck{X}_{t})=
\lim_{u\to +\infty}\check{\lambda}_{0}(y_{\rm bif})
(1+2\widecheck{\Lambda}_{u}(y_{\rm bif}))^{-1}=0,
\end{displaymath}
as $\lim_{u\to +\infty}\widecheck{\Lambda}_{u}(y_{\rm bif})=+\infty$ 
(see Section 4 in \cite{HuWarren00BBFlow}). 

Also note that if one sets
$\widecheck{S}_{t} = \widecheck{\Psi}_{u(t)}\circ\widecheck{S}_{0}$,
then
$\widecheck{S}_{t}(\widecheck{X}_{t}) = B_{u(t)}$,
and
$(\widecheck{S}_{t\wedge\widecheck{T}}(\widecheck{X}_{t\wedge\widecheck{T}}))_{t\geq 0}$
is a local martingale.
Moreover,
\begin{displaymath}
\dfrac{\partial }{\partial x}\widecheck{S}_{t}(x)
=\check{\lambda}_{0}(x)^{-1}
(1+2\widecheck{\Lambda}_{u(t)}(\widecheck{S}_{0}(x)))
= \check{\lambda}_{t}(x)^{-1}.
\end{displaymath}
So, $\widecheck{S}_{t}$ is a time-dependent change of scale
indeed satisfying the properties postulated previously
in our heuristic.

\subsection*{Statement of the results}

Now, let us see why $(\widecheck{X}_{t},\check{\lambda}_{t})$ can be interpreted as solution to the equation \eqref{MainEq}, with initial condition $(x_{0},\check{\lambda}_{0})$. We will give an explanation in terms of discrete approximations.
Let $\J^{(n)}=2^{-n}\mathbb{Z}\cap I$. Let $\widecheck{X}^{(n)}_{t}$ be a continuous time discrete space self-interacting nearest neighbor jump process on $\J^{(n)}$, defined by the jumps rates from $x$ to
$x+\sigma 2^{-n}$, $\sigma \in\lbrace -1,1\rbrace$, at time $t$, equal to
\begin{equation}
\label{EqRate}
2^{2n-1}\dfrac{\check{\lambda}_{t}^{(n)}(x+\sigma 2^{-n})^{\frac{1}{2}}}
{\check{\lambda}_{t}^{(n)}(x)^{\frac{1}{2}}},
\end{equation}
where
\begin{displaymath}
\check{\lambda}_{t}^{(n)}(x)=
\check{\lambda}_{0}(x)-2
\check{\ell}_{t}^{(n)}(x),
\qquad
\check{\ell}_{t}^{(n)}(x)=2^{n}
\int_{0}^{t}\1_{\widecheck{X}^{(n)}_{s}=x} ds.
\end{displaymath}
Let
\begin{displaymath}
\widecheck{T}^{(n)}_{\varepsilon}=\sup\lbrace t\geq 0\vert 
\check{\lambda}_{t}^{(n)}(\widecheck{X}^{(n)}_{t})>\varepsilon\rbrace,
\varepsilon>0,
~~
t^{(n)}_{\partial\J^{(n)}}
=\inf\lbrace t\geq 0\vert 
\widecheck{X}^{(n)}_{t}\in\lbrace\min\J^{(n)},
\max \J^{(n)}\rbrace\rbrace.
\end{displaymath}
We introduce the stopping time 
$t^{(n)}_{\partial\J^{(n)}}$
to avoid considering what happens after
$\widecheck{X}^{(n)}_{t}$ hits the boundary of
the domain $\J^{(n)}$.

If there were no self-interaction, that is to say in \eqref{EqRate}
$\check{\lambda}_{t}^{(n)}$ were replaced by
$\check{\lambda}_{0}$, the process would converge in law as $n\to +\infty$ to  a solution of the SDE
\begin{displaymath}
d X_{t}= d W_{t} + 
\dfrac{1}{2}\partial_{x}\log(\check{\lambda}_{0}(x))
\Big\vert_{x= X_{t}} dt.
\end{displaymath}
In our case with self-interaction, we have the following:

\begin{thm}
\label{ThmSIConv}
With the notations above, for all $\varepsilon>0$ 
the family of process
$$(\widecheck{T}^{(n)}_{\varepsilon}\wedge 
t^{(n)}_{\partial\J^{(n)}},
\widecheck{X}^{(n)}_{t\wedge\widecheck{T}^{(n)}_{\varepsilon}
\wedge 
t^{(n)}_{\partial\J^{(n)}}},
\check{\lambda}_{t\wedge
\widecheck{T}^{(n)}_{\varepsilon}
\wedge 
t^{(n)}_{\partial\J^{(n)}}
}^{(n)}(x))_{x\in \J^{(n)}, t\geq 0}$$ converges in law as 
$n\to +\infty$ to
$$(\widecheck{T}_{\varepsilon},\widecheck{X}_{t\wedge\widecheck{T}_{\varepsilon}},
\check{\lambda}_{t\wedge\widecheck{T}_{\varepsilon}}~(x))
_{x\in I, t\geq 0},$$
where $\widecheck{X}_{t}$ is given by Definition \ref{DefMain}
and
\begin{displaymath}
\widecheck{T}_{\varepsilon}=\sup\lbrace t\geq 0\vert 
\check{\lambda}_{t}(\widecheck{X}_{t})>\varepsilon\rbrace,
\end{displaymath}
provided that $\widecheck{X}^{(n)}_{0}$ converges to
$\widecheck{X}_{0}$.
In particular, 
\begin{displaymath}
t^{(n)}_{\partial\J^{(n)}}>\widecheck{T}^{(n)}_{\varepsilon}
\end{displaymath}
with probability converging to $1$.
The convergence in law is for the topology of uniform convergence on compact subsets of $I\times [0,+\infty)$. The spatial processes on 
$\J^{(n)}$ are considered to be linearly interpolated outside 
$\J^{(n)}$.
\end{thm}

Next we state how our self-repelling diffusion is related to the inversion of the Ray-Knight identity of Theorem
\ref{ThmRK}.

\begin{thm}
\label{ThmRKInv}
Let $a>0$ and $(\phi^{(a)}(x))_{x\in\R}$ be a massless Gaussian free field on $\R$ conditioned to be $a$ at $x=0$. 
Let $I(\phi^{(a)})$ be the connected component of $0$ in 
$\lbrace x\in\R\vert \phi^{(a)}(x)>0\rbrace$. For $x\in I(\phi^{(a)})$, set
$\check{\lambda}_{0}^{\ast}(x)=\phi^{(a)}(x)^{2}$. Then
a.s. $\check{\lambda}^{\ast}_{0}$ satisfies the condition 
\eqref{EqCond}. Let
$(\widecheck{X}^{\ast}_{t},\check{\lambda}^{\ast}_{t}(x))
_{x\in I(\phi^{(a)}), 0\leq t\leq \widecheck{T}^{\ast}}$ be the process, distributed conditionally on 
$(\phi^{(a)}(x))_{x\in\R}$, as the self repelling diffusion on 
$I(\phi^{(a)})$, starting from $0$, 
with initial occupation profile 
$\check{\lambda}_{0}^{\ast}$, following Definition \ref{DefMain}. Let be the triple
\begin{displaymath}
(\phi^{(0)}(x)^{2},
\beta_{t},
\phi^{(a)}(x)^{2})_{x\in\R, 0\leq t\leq \tau^{\beta}_{a^{2}/2}},
\end{displaymath}
jointly distributed as in 
the Ray-Knight coupling 
(Definition \ref{Def RK coupl}).
Let be
\begin{displaymath}
\widecheck{T}^{\beta,a}=\tau^{\beta}_{a^{2}/2}-
\sup\lbrace t\in [0,\tau^{\beta}_{a^{2}/2}]\vert
\phi^{(0)}(\beta_{t})=0~\text{and}~
\forall s\in[0,t),\beta_{s}\neq\beta_{t}\rbrace.
\end{displaymath}
Then the couple
\begin{displaymath}
(\widecheck{X}^{\ast}_{t},
\phi^{(a)}(x)^{2})_{x\in\R, 0\leq t\leq \widecheck{T}^{\ast}}
\end{displaymath}
has the same distribution as 
\begin{displaymath}
(\beta_{\tau^{\beta}_{a^{2}/2}-t},
\phi^{(a)}(x)^{2})_{x\in\R, 0\leq t\leq \widecheck{T}^{\beta,a}}.
\end{displaymath}
\end{thm}

The notation 
$(\widecheck{X}^{\ast}_{t},\check{\lambda}^{\ast}_{t}(x),
\widecheck{T}^{\ast})$ is 
reserved to the case of the initial occupation profile
$\check{\lambda}^{\ast}_{0}(x)=\phi^{(a)}(x)^{2}$, so as to
to avoid confusion with the case of generic $\check{\lambda}_{0}$.

Note that Theorem \ref{ThmRKInv} also trivially implies that
the triple
\begin{equation}
\label{Eq triplet}
(\widecheck{X}^{\ast}_{t},
\phi^{(a)}(x)^{2},
\check{\lambda}_{t}^{\ast}(x))_{x\in\R, 0\leq t\leq \widecheck{T}^{\ast}}
\end{equation}
has the same distribution as
\begin{equation}
\label{Eq triplet TR}
(\beta_{\tau^{\beta}_{a^{2}/2}-t},
\phi^{(a)}(x)^{2},
\phi^{(0)}(x)^{2}+ 2\ell^{\beta}_{\tau^{\beta}_{a^{2}/2}-t}(x)
)_{x\in\R, 0\leq t\leq \widecheck{T}^{\beta,a}}.
\end{equation}
Moreover, since the law of
$(\beta_{\tau^{\beta}_{a^{2}/2}-t})_{0\leq t\leq \tau^{\beta}_{a^{2}/2}}$
is the same as that of
$(\beta_{t})_{0\leq t\leq \tau^{\beta}_{a^{2}/2}}$,
we have that \eqref{Eq triplet} is also distributed as
\begin{displaymath}
(\beta_{t},
\phi^{(a)}(x)^{2},
\phi^{(0)}(x)^{2} + 
2\ell^{\beta}_{\tau^{\beta}_{a^{2}/2}}(x)
- 2\ell^{\beta}_{t}(x)
)_{x\in\R, 0\leq t\leq T^{\beta,a}},
\end{displaymath}
where
\begin{displaymath}
T^{\beta,a} =
\inf\lbrace t\in [0,\tau^{\beta}_{a^{2}/2}]\vert
\phi^{(0)}(\beta_{t})=0~\text{and}~
\forall s\in(t,\tau^{\beta}_{a^{2}/2}],\beta_{s}\neq\beta_{t}\rbrace.
\end{displaymath}
We prefer the time-reversed presentation
\eqref{Eq triplet TR} as we imagine the Brownian path
$(\beta_{t})_{0\leq t\leq \tau^{\beta}_{a^{2}/2}}$
being reconstructed from its end 
by starting from the final condition
$(\phi^{(a)}(x)^{2})_{x\in\R}$
for $(\phi^{(0)}(x)^{2}+2\ell^{\beta}_{t}(x))_{x\in\R}$.

One could also extend the definition of the self-repelling diffusion to metric graphs, where it is again related to the inversion of the Ray-Knight identity. The proof would be essentially the same as for an interval. We won't detail it here. A metric graph is obtained by replacing in an undirected graph each edge by a continuous line segment of certain length, corresponding to the resistance of the edge.
For background on Markovian (non-selfinteracting) diffusions on metric
graphs, see \cite{BaxterChacon1984DiffusionsNetworks,EnriquezKifer2001BMGraphs,
FW2016PhD}. The Gaussian free field on metric graphs was introduced in
\cite{Lupu2016Iso}, in relation with isomorphism theorems.

\subsection*{Other works on self-interacting diffusions in dimension one}

Now let us review some other works on self-interacting diffusions in dimension one and their relations to ours. Two other Bass-Burdzy flows appeared in construction of self-interacting diffusions. First, in \cite{Warren05BBFlow} it was shown that the flow of solutions to 
\begin{displaymath}
\dfrac{dY_{u}}{du} = \1_{Y_{u}>B_{u}}
\end{displaymath}
was related to the Brownian first passage bridge conditioned by its family of local times and to the Brownian burglar \cite{WarrenYor1998Burglar}.
There the problem is similar to ours, i.e. constructing a Brownian motion with some conditioning on its family of local times, yet it is different and the processes obtained are different. Then, in \cite{LST2018LRM} we constructed a linearly reinforced diffusion on 
$\R$ out of the flow of solutions to
\begin{displaymath}
\dfrac{dY_{u}}{du} = -\1_{Y_{u}>B_{u}} + \1_{Y_{u}<B_{u}},
\end{displaymath}
that is to say the signs are opposite to those in \eqref{EqBB}.
The reinforced diffusion in \cite{LST2018LRM} can be considered as a dual of the self-repelling diffusion in the present paper.

Our self-repelling diffusion is different from the Brownian polymer models studied in \cite{DurrettRogers92BrownianPolymer,CranstonLeJan952Case,
CranstonMountford96BrownianPolymer,Pemantle07Survey,
TarresTothValko2012BrownianPolymer,
BenaimCiotirGauthier2015SelfRep,
Gauthier2018SelfRep}, as here the interaction of the moving particle with the occupation profile occurs locally, at zero range, and is not an average over positive ranges. In other words, we do not mollify the occupation profile prior to taking its derivative. Also, for that reason, we do not expect our process to be a semi-martingale as the above Brownian polymer models, but only a Dirichlet process in the sense of Föllmer \cite{Follmer81DirichletProc}, admitting an adapted decomposition 
into a continuous local martingale and zero quadratic variation drift, with the drift term not necessarily of bounded variation
(see also \cite{HuWarren00BBFlow} and
\cite{LST2018LRM}). Our process is also different from the true self-repelling motion (TSRM) introduced by Tóth and Werner in \cite{TothWerner98TSRM}, as our process, unlike the TSRM, has the Hölder regularity of a Brownian motion and does not exhibit a $2/3$ scaling exponent.
We do not know if our process is related to the continuum directed random polymer introduced by Alberts, Khanin and Quastel in \cite{AlbertsKhaninQuastel14DirectedRandomPolymer}.

\subsection*{Organization of the article}

We will first prove Theorem \ref{ThmRKInv}, as well as Theorem
\ref{ThmSIConv} in the particular case when the initial occupation profile is random, given by the square of the free field. The proof relies on the restriction of continuous processes to discrete subsets and on taking the limit of discrete space processes inverting Ray-Knight in the discrete setting. On one hand these discrete space processes are embedded into a standard Brownian motion. So the limit exists \textit{a priori} and is a Brownian path. On the other hand one can construct out of the discrete space processes discrete analogues of the divergent Bass-Burdzy flow converging to the latter. Further, the proof of Theorem \ref{ThmSIConv} for general occupation profile will follow out of a path transformation as in Proposition \ref{PropElementary} (1). 

The reason why we do not proceed directly to the proof of
Theorem \ref{ThmSIConv} for general occupation profile is that we need tightness and need the ratio
\begin{displaymath}
\dfrac{\check{\lambda}^{(n)}
_{t\wedge \widecheck{T}^{(n)}_{\varepsilon}}
(x+2^{-n})}
{\check{\lambda}^{(n)}
_{t\wedge \widecheck{T}^{(n)}_{\varepsilon}}(x)}
\end{displaymath}
to converge to $1$ as $n\to +\infty$, uniformly in
$(x,t)$ on compact subsets. For
$\check{\lambda}^{\ast}_{0}=(\phi^{(a)})^{2}$ this is achieved by 
embedding the discrete space self-repelling processes into a Brownian motion.

Our article is organized as follows. In Section \ref{SecProperties} we will give some properties of our self-repelling diffusion. 
In Section \ref{SecDiscr} we will recall how the self-repelling jump processes of Theorem \ref{ThmSIConv} appear in the inversion of the Ray-Knight identity on discrete subsets of $\R$. This is a result obtained in \cite{LST2017InvRK}. Using this we will prove in Section \ref{SecBBGFF}
the Theorem \ref{ThmRKInv} and the particular case of Theorem \ref{ThmSIConv} when 
$\check{\lambda}_{0}(x)=
\check{\lambda}_{0}^{\ast}(x)=\phi^{(a)}(x)^{2}$.
In Section \ref{SecBBgeneral} we will prove 
Theorem \ref{ThmSIConv} in general.

\section{Elementary properties of the self-repelling diffusion}
\label{SecProperties}

First, we give some elementary properties of
$(\widecheck{X}_{t},\check{\lambda}_{t}(x))_{x\in I, 0\leq t\leq 
\widecheck{T}}$. They are straighforward and come without proofs. For proofs of analogous statements, see Proposition 2.4 in \cite{LST2018LRM}.

\begin{prop}
\label{PropElementary}
(1) Let $I$ and $I^{\bullet}$ be two open subintervals of $\R$. 
Let be $\check{\lambda}_{0}$ and
$\check{\lambda}_{0}^{\bullet}$ two admissible initial occupation profiles on $I$, respectively $I^{\bullet}$, and 
$(\widecheck{X}_{t},\check{\lambda}_{t}(x))
_{x\in I, 0\leq t\leq \widecheck{T}}$, 
$(\widecheck{X}^{\bullet}_{t},\check{\lambda}^{\bullet}_{t}(x))
_{x\in I^{\bullet}, 0\leq t\leq \widecheck{T}^{\bullet}}$ the corresponding self-repelling diffusions, starting from $x_{0}\in I$, respectively 
$x^{\bullet}_{0}\in I^{\bullet}$. One can go from one to the other by a deterministic path transformation. More precisely, let
\begin{displaymath}
\widecheck{S}_{0}(x)=
\int_{x_{0}}^{x}\check{\lambda}_{0}(r)^{-1} dr,
x\in I,
\qquad
\widecheck{S}_{0}^{\bullet}(x)=
\int_{x_{0}^{\bullet}}^{x}\check{\lambda}^{\bullet}_{0}(r)^{-1} dr,
x\in I^{\bullet}.
\end{displaymath}
Let $t\mapsto \theta^{\bullet}(t)$ be the change of time
\begin{displaymath}
d\theta^{\bullet}(t)=
\check{\lambda}_{0}^{\bullet}(
(\widecheck{S}_{0}^{\bullet})^{-1}\circ\widecheck{S}_{0}
(\widecheck{X}_{t}))^{2}
\check{\lambda}_{0}(\widecheck{X}_{t})^{-2}
dt.
\end{displaymath}
Then then process
$((\widecheck{S}_{0}^{\bullet})^{-1}\circ\widecheck{S}_{0}
(\widecheck{X}_{(\theta^{\bullet})^{-1}(t)}),
\check{\lambda}_{(\theta^{\bullet})^{-1}(t)}
(\widecheck{S}_{0}^{-1}\circ\widecheck{S}^{\bullet}_{0}(x)))
_{x\in I^{\bullet}, 0\leq t\leq \theta^{\bullet}(\widecheck{T})}$
has the same law as 
$(\widecheck{X}^{\bullet}_{t},\check{\lambda}^{\bullet}_{t}(x))
_{x\in I^{\bullet}, 0\leq t\leq \widecheck{T}^{\bullet}}$.

(2)(Strong Markov property) For any $T$ stopping time for the natural filtration of
$(\widecheck{X}_{t\wedge \widecheck{T}})_{t\geq 0}$, such that
$T<\widecheck{T}$ a.s., the process
$(\widecheck{X}_{T+t},\check{\lambda}_{T+t}(x))_{x\in I,0\leq t\leq \widecheck{T}-T}$, conditional on the past before $T$, is a self-repelling diffusion with inital occupation profile
$\check{\lambda}_{T}$.

(3) Let $a<b\in I$ such that $a<x_{0}<b$. Let $t_{a,b}$ be the first time $t$ that
$\widecheck{X}_{t}$ hits $a$ or $b$. We consider the 
stopping time $t_{a,b}\wedge\widecheck{T}$.
Given a Brownian motion $(B_{u})_{u\geq 0}$.
Let $U^{\uparrow \widecheck{S}_{0}(b)}\in (0,+\infty]$ 
the first time $B_{u}-u$ hits $\widecheck{S}_{0}(b)$, whenever this happens.
Let $U^{\downarrow \widecheck{S}_{0}(a)}\in (0,+\infty]$ the first time $B_{u}+u$ hits 
$\widecheck{S}_{0}(a)$, whenever this happens. Then
\begin{displaymath}
\mathbb{P}(t_{a,b}<\widecheck{T},\widecheck{X}_{t_{a,b}}=b)=
\mathbb{P}(U^{\uparrow \widecheck{S}_{0}(b)}<U^{\downarrow \widecheck{S}_{0}(a)}),
\qquad
\mathbb{P}(t_{a,b}<\widecheck{T},\widecheck{X}_{t_{a,b}}=a)=
\mathbb{P}(U^{\downarrow \widecheck{S}_{0}(a)}<U^{\uparrow \widecheck{S}_{0}(b)}),
\end{displaymath}
\begin{displaymath}
\mathbb{P}(t_{a,b}>\widecheck{T})=
\mathbb{P}(U^{\uparrow \widecheck{S}_{0}(b)}=U^{\downarrow \widecheck{S}_{0}(a)}
=+\infty).
\end{displaymath}
In particular,
\begin{displaymath}
\mathbb{P}(t_{a,b}<\widecheck{T},\widecheck{X}_{t_{a,b}}=b)\geq
\mathbb{P}(t_{a,b}<\widecheck{T},\widecheck{X}_{t_{a,b}}=a)
\end{displaymath}
if and only if $$
\int_{x_{0}}^{b}\check{\lambda}_{0}(r)^{-1} dr
=\widecheck{S}_{0}(b)\leq\vert\widecheck{S}_{0}(a)\vert=
\int_{a}^{x_{0}}\check{\lambda}_{0}(r)^{-1}
dr .$$
\end{prop}

Next we state that our self-repelling diffusion depends continuously on the initial occupation profile. 

\begin{lemma}
\label{LemContinuity}
Let $(I_{k})_{n\geq 0}$ be a sequence of open subintervals of $\R$ such that
\begin{displaymath}
\lim_{n\to +\infty}\inf I_{k} = \inf I \in [-\infty,+\infty),
\qquad
\lim_{n\to +\infty}\sup I_{k} = \sup I \in (-\infty,+\infty].
\end{displaymath}
On each $I_{k}$ we consider
$\check{\lambda}_{0}^{I_{k}}$ an admissible occupation profile, and we assume that for all $K$ compact subset of $I$,
\begin{displaymath}
\lim_{k\to +\infty}
\sup_{K\cap I_{k}}\vert\check{\lambda}_{0}
-\check{\lambda}_{0}^{I_{k}}\vert=0.
\end{displaymath}
Let $(\widecheck{X}^{I_{k}}_{t},\check{\lambda}_{t}^{I_{k}}(x))
_{x\in I_{k},0\leq t\leq \widecheck{T}^{I_{k}}}$ be the self-repelling diffusion on $I_{k}$ with initial occupation profile $\check{\lambda}_{0}^{I_{k}}$. We assume that 
$\lim_{k\to +\infty} \widecheck{X}^{I_{k}}_{0}=\widecheck{X}_{0}\in I$.
Then, as $k\to +\infty$,
\begin{displaymath}
(\widecheck{X}^{I_{k}}_{t\wedge\widecheck{T}^{I_{k}}},\check{\lambda}_{t\wedge \widecheck{T}^{I_{k}}}^{I_{k}}(x))
_{x\in I_{k},t\geq 0}
\end{displaymath}
converges in law to
\begin{displaymath}
(\widecheck{X}_{t\wedge\widecheck{T}},
\check{\lambda}_{t\wedge \widecheck{T}}(x))
_{x\in I,t\geq 0},
\end{displaymath}
where the convergence is for the uniform topology in
$t\in [0,+\infty)$, and uniform on compact subsets of $I$ 
for $x$.
\end{lemma}

\begin{proof}
This is an immediate consequence of
Definition \ref{DefMain} and Proposition \ref{PropElementary} (1). If moreover all of the processes
$(\widecheck{X}^{I_{k}}_{t},\check{\lambda}_{t}^{I_{k}}(x))
_{x\in I_{k},0\leq t\leq \widecheck{T}^{I_{k}}}$ and
$(\widecheck{X}_{t\wedge\widecheck{T}},
\check{\lambda}_{t\wedge \widecheck{T}}(x))
_{x\in I,t\geq 0}$
are constructed of the same driving Brownian motion
$(B_{u})_{u\geq 0}$ (Definition \ref{DefMain}), then the convergence is
a.s.
Indeed, one uses the same process
$(\check{\xi}_{u})_{u\geq 0}$, and the change of scale and change of time functions involved in the construction of
$(\widecheck{X}^{I_{k}}_{t},\check{\lambda}_{t}^{I_{k}}(x))
_{x\in I_{k},0\leq t\leq \widecheck{T}^{I_{k}}}$
converge as $k\to +\infty$.
\end{proof}

\section{Inversion of the Ray-Knight identity on a discrete subset}
\label{SecDiscr}

Consider the triple 
\begin{displaymath}
(\phi^{(0)}(x)^{2},\beta_{t},
\phi^{(a)}(x)^{2})_{x\in\R, 0\leq t\leq \tau^{\beta}_{a^{2}/2}},
\end{displaymath}
jointly distributed as in the Ray-Knight 
coupling (Definition \ref{Def RK coupl}).
Let $I(\phi^{(a)})$ be the connected component of $0$ in 
$\lbrace x\in\R\vert \phi^{(a)}(x)>0\rbrace$.

Let $\ell^{\beta}_{t}(x)$ be the family of local times of
the Brownian motion $\beta_{t}$.
Let $\J^{\bullet}$ be a countable discrete subset of $\R$, containing $0$, unbounded in both directions. Consider the change of time
\begin{displaymath}
Q^{\J^{\bullet},\beta}(t)=
\sum_{x\in\J^{\bullet}}\ell^{\beta}_{t}(x).
\end{displaymath}
Define $X^{\J^{\bullet}}_{q}
=\beta_{(Q^{\J^{\bullet},\beta})^{-1}(q)}$, where 
$(Q^{\J^{\bullet},\beta})^{-1}$ is the right-continuous inverse
of $Q^{\J^{\bullet},\beta}$. It is a nearest neighbor Markov jump process on $\J^{\bullet}$, with jump rate from a vertex $x_{1}$ to a neighbor $x_{2}$ equal to the conductance
\begin{equation}
\label{EqConductance}
C(x_{1},x_{2})=\dfrac{1}{2\vert x_{2}-x_{1}\vert}.
\end{equation}

Given $x\in \J^{\bullet}$, $\lambda^{\J^{\bullet}}_{q}(x)$ will denote
\begin{displaymath}
\lambda^{\J^{\bullet}}_{q}(x)=\phi^{(0)}(x)^{2}+
2\ell^{\beta}_{(Q^{\J^{\bullet},\beta})^{-1}(q)}(x)=
\lambda^{\J^{\bullet}}_{0}(x)+2\int_{0}^{q}\1_{X^{\J^{\bullet}}_{r}=x} dr.
\end{displaymath}

Next, for $q\geq 0$, $\O^{\J^{\bullet}}_{q}$ will denote a function from pairs of neighbor vertices in $\J^{\bullet}$ to $\{0,1\}$. Given 
$x_{1}<x_{2}$ two neighbors in $\J^{\bullet}$, we will say that the edge
$\{x_{1},x_{2}\}$ is \textit{open} (at time $q$) if 
$\O^{\J^{\bullet}}_{q}(\{x_{1},x_{2}\})=1$ and \textit{closed} if
$\O^{\J^{\bullet}}_{q}(\{x_{1},x_{2}\})=0$. 
$\O^{\J^{\bullet}}_{q}(\{x_{1},x_{2}\})$ is defined as follows:
\begin{displaymath}
\O^{\J^{\bullet}}_{0}(\{x_{1},x_{2}\})= \1_{\phi^{(0)}(x)^{2} 
\text{ has no zeroes on } [x_{1},x_{2}]},
\end{displaymath}
\begin{displaymath}
\O^{\J^{\bullet}}_{q}(\{x_{1},x_{2}\})= \1_{\phi^{(0)}(x)^{2}
+2\ell^{\beta}_{(Q^{\J^{\bullet},\beta})^{-1}(q)}(x)
\text{ has no zeroes on } [x_{1},x_{2}]},
\end{displaymath}
By construction, $(\O^{\J^{\bullet}}_{q})_{t\geq 0}$ is a family non-decreasing in $q$.

Next we state that the joint process 
$(X^{\J^{\bullet}}_{q},\lambda^{\J^{\bullet}}_{q},\O^{\J^{\bullet}}_{q})_{q\geq 0}$ is Markovian and give the transitions rates. For details we refer to Theorem 8 in
\cite{LST2017InvRK}.

\begin{prop}
\label{PropDirect}
$(X^{\J^{\bullet}}_{q},\lambda^{\J^{\bullet}}_{q},\O^{\J^{\bullet}}_{q})_{q\geq 0}$ is a Markov process. Let $x_{1}$ and $x_{2}$ be two neighbors in $\J^{\bullet}$. If
$X^{\J^{\bullet}}_{q}=x_{1}$, then:
\begin{itemize}
\item $X^{\J^{\bullet}}_{q}$ jumps to $x_{2}$ with rate
$(2\vert x_{2}-x_{1}\vert)^{-1}$.
$\O^{\J^{\bullet}}_{q}(\{x_{1},x_{2}\})$ is then set to $1$ (if it was not already).
\item In case $\O^{\J^{\bullet}}_{q}(\{x_{1},x_{2}\})=0$, 
$\O^{\J^{\bullet}}_{q}(\{x_{1},x_{2}\})$ is set to $1$
without $X^{\J^{\bullet}}_{q}$ jumping with rate
\begin{equation}
\label{EqR2}
\dfrac{1}{2\vert x_{2}-x_{1}\vert}
\dfrac{\lambda^{\J^{\bullet}}_{q}(x_{2})^{\frac{1}{2}}}{\lambda^{\J^{\bullet}}_{q}(x_{1})^{\frac{1}{2}}}
\exp(-\vert x_{2}-x_{1}\vert^{-1}
\lambda^{\J^{\bullet}}_{q}(x_{1})^{\frac{1}{2}}
\lambda^{\J^{\bullet}}_{q}(x_{2})^{\frac{1}{2}}).
\end{equation}
\end{itemize}
For $x\in \J^{\bullet}$,
\begin{displaymath}
\lambda^{\J^{\bullet}}_{q}(x)=
\lambda^{\J^{\bullet}}_{0}(x)+2\int_{0}^{q}\1_{X^{\J^{\bullet}}_{r}=x} dr.
\end{displaymath}
\end{prop}

\begin{proof}
If $X^{\J^{\bullet}}_{q}$ jumps through the edge $\{x_{1},x_{2}\}$, then
$\beta_{t}$ crosses the interval delimited by $x_{1}$ and $x_{2}$, and then the local time of $\beta_{t}$ on this interval is positive, and thus
$\O^{\J^{\bullet}}_{q}(\{x_{1},x_{2}\})=1$ after the jump.

As described in Section 2 in \cite{LST2017InvRK} and in particular in Theorem 8, the conditional probability that
$\O^{\J^{\bullet}}_{q}(\{x_{1},x_{2}\})=0$, given 
$(X^{\J^{\bullet}}_{r},\lambda^{\J^{\bullet}}_{r}(x))
_{x\in\J^{\bullet}, 0\leq r\leq q}$, 
and that
$\O^{\J^{\bullet}}_{0}(\{x_{1},x_{2}\})=0$,
and that
$X^{\J^{\bullet}}_{r}$ has not crossed the edge $\{x_{1},x_{2}\}$ before time $q$, equals
\begin{displaymath}
\exp(C(x_{1},x_{2})\lambda^{\J^{\bullet}}_{0}(x_{1})^{\frac{1}{2}}
\lambda^{\J^{\bullet}}_{0}(x_{2})^{\frac{1}{2}}-C(x_{1},x_{2})
\lambda^{\J^{\bullet}}_{q}(x_{1})^{\frac{1}{2}}
\lambda^{\J^{\bullet}}_{q}(x_{2})^{\frac{1}{2}}),
\end{displaymath}
where $C(x_{1},x_{2})$ is given by \eqref{EqConductance}.
Thus, the rate \eqref{EqR2} is obtained as
\begin{displaymath}
\lim_{\Delta q \to 0^{+}} \dfrac{1}{\Delta q}
\bigg(
1-
\dfrac{\exp(
C(x_{1},x_{2})\lambda^{\J^{\bullet}}_{0}(x_{1})^{\frac{1}{2}}
\lambda^{\J^{\bullet}}_{0}(x_{2})^{\frac{1}{2}}
-C(x_{1},x_{2})
(\lambda^{\J^{\bullet}}_{q}(x_{1})+2\Delta q)^{\frac{1}{2}}
\lambda^{\J^{\bullet}}_{q}(x_{2})^{\frac{1}{2}})}
{\exp(
C(x_{1},x_{2})\lambda^{\J^{\bullet}}_{0}(x_{1})^{\frac{1}{2}}
\lambda^{\J^{\bullet}}_{0}(x_{2})^{\frac{1}{2}}
-C(x_{1},x_{2})
\lambda^{\J^{\bullet}}_{q}(x_{1})^{\frac{1}{2}}
\lambda^{\J^{\bullet}}_{q}(x_{2})^{\frac{1}{2}})}
\bigg).
\qedhere
\end{displaymath}
\end{proof}

The fields $\phi^{(0)}$ and $\phi^{(a)}$ restricted to $\J^{\bullet}$ are discrete Gaussian free fields on $\J^{\bullet}$. The triple
\begin{displaymath}
(\phi^{(0)}(x)^{2},X^{\J^{\bullet}}_{q},
\phi^{(a)}(x)^{2})_{x\in\J^{\bullet}, 0\leq q\leq Q^{\J^{\bullet},\beta}(\tau^{\beta}_{a^{2}/2})}
\end{displaymath}
satisfies the Ray-Knight identity on the discrete network
$\J^{\bullet}$. So in \cite{LST2017InvRK} one can find a procedure inverting this Ray-Knight identity in the discrete setting. It corresponds to a time reversal of the process of Proposition \ref{PropDirect} from stopping time $Q^{\J^{\bullet},\beta}(\tau^{\beta}_{a^{2}/2})$.
This is explained in Section 3 in \cite{LST2017InvRK}, in particular in Proposition 3.4 there.
Note that the introduction of the variables 
$\O^{\J^{\bullet}}_{q}$, as in \cite{LST2017InvRK}, 
allows for a simpler expression of the inversion procedure.
This is related to the fact that
$\phi^{(0)}(x_{1})\phi^{(0)}(x_{2})>0$ 
whenever $\O^{\J^{\bullet}}_{0}(\{x_{1},x_{2}\})=1$.
The inversion procedure that does not keep track of
the variables $\O^{\J^{\bullet}}_{q}$ is presented in
\cite{ST2016InvRK}, and it involves more complicated expressions
with conditional expectations of relative signs.

Let $\J$ be a finite subset of $\R$ containing $0$.
Let us consider the continuous time discrete space self-repelling nearest neighbor jump process on $\J$, which has been introduced in \cite{LST2017InvRK}. Let 
$\check{\lambda}^{\J}_{0}$ be a positive function on 
$\J$. We consider the process 
$(\widecheck{X}_{q}^{\J},
\check{\lambda}_{q}^{\J}(x))
_{x\in \J,q\geq 0}$, where 
$\widecheck{X}_{q}^{\J}$ is a nearest neighbor jump process on 
$\J$, starting from $0$, 
with time-dependent jump rates from $x_{1}$ to a neighbor 
$x_{2}$ in $\J$ given by
\begin{equation}
\label{EqDSRP1}
\dfrac{1}{2\vert x_{2}-x_{1}\vert}
\dfrac{\check{\lambda}_{q}^{\J}
(x_{2})^{\frac{1}{2}}}
{\check{\lambda}_{q}^{\J}
(x_{1})^{\frac{1}{2}}},
\end{equation}
and
\begin{equation}
\label{EqDSRP2}
\check{\lambda}_{q}^{\J}(x)=
\check{\lambda}_{0}^{\J}(x)-2\int_{0}^{q}
\1_{\widecheck{X}_{r}^{\J}=x} dr.
\end{equation}
Let be $\widecheck{\mathcal{Q}}^{\J}$ be a the random time coupled to
$(\widecheck{X}_{q}^{\J},
\check{\lambda}_{q}^{\J}(x))_{x\in \J, q\geq 0}$ in the following way. 
If $\J$ is reduced to $\lbrace 0\rbrace$, then we set
$\widecheck{\mathcal{Q}}^{\J}=0$. Otherwise,
$\widecheck{\mathcal{Q}}^{\J}$ is the first time $q$ when the integral
\begin{equation}
\label{EqDSRP3}
\int_{0}^{q}
\sum_{x\stackrel{\J}{\sim}\widecheck{X}_{r}^{\J}}
\bigg(
\dfrac{
\check{\lambda}_{r}^{\J}(x)^{\frac{1}{2}}}
{
\vert\widecheck{X}_{r}^{\J}-x\vert
\check{\lambda}_{r}^{\J}(\widecheck{X}_{r}^{\J})
^{\frac{1}{2}}}
(\exp(\vert\widecheck{X}_{r}^{\J}-x \vert^{-1}
\check{\lambda}_{r}^{\J}(x)^{\frac{1}{2}}
\check{\lambda}_{r}^{\J}(\widecheck{X}_{r}^{\J})
^{\frac{1}{2}})-1)^{-1}
\bigg) dr
\end{equation}
hits an independent exponential random variable of mean $1$.
The notation $x\stackrel{\J}{\sim}\widecheck{X}_{r}^{\J}$ means that
$x$ is a neighbor of $\widecheck{X}_{r}^{\J}$ in $\J$. We will further explain where the definition of $\widecheck{\mathcal{Q}}^{\J}$ comes from.
Note that a.s. the time $\widecheck{\mathcal{Q}}^{\J}$ fires before one of the $\check{\lambda}_{q}^{\J}(x)$ reaches $0$. This is due to the fact that
\begin{displaymath}
\forall K>0,~~\int_{0}\dfrac{1}{r^{1/2}}
(\exp(K r^{1/2})-1)^{-1} dr = +\infty.
\end{displaymath}

Next we describe the process 
$(\widebreve{X}^{\J^{\bullet}}_{q},\breve{\lambda}^{\J^{\bullet}}_{q},
\widebreve{\O}^{\J^{\bullet}}_{q})_{q\geq 0}$ 
introduced in Section 3.3 in \cite{LST2017InvRK}.
$\widebreve{\O}^{\J^{\bullet}}_{q}$ is a function from pairs of neighbor vertices in $\J^{\bullet}$ to $\{0,1\}$. Given $x_{1}<x_{2}$ two neighbors in $\J^{\bullet}$, we set
\begin{displaymath}
\widebreve{\O}^{\J^{\bullet}}_{0}(\{x_{1},x_{2}\})= 
\1_{\phi^{(a)}(x)^{2} 
\text{ has no zeroes on } [x_{1},x_{2}]}.
\end{displaymath}
$\widebreve{X}^{\J^{\bullet}}_{q}$ is a nearest neighbor jump process on
$\J^{\bullet}$. $\widebreve{X}^{\J^{\bullet}}_{0}=0$.
For $x\in \J^{\bullet}$, 
\begin{displaymath}
\breve{\lambda}^{\J^{\bullet}}_{q}(x)=\phi^{(a)}(x)^{2}
-2\int_{0}^{q}\1_{\widebreve{X}^{\J^{\bullet}}_{r}=x} dr
=\breve{\lambda}^{\J^{\bullet}}_{0}(x)
-2\int_{0}^{q}\1_{\widebreve{X}^{\J^{\bullet}}_{r}=x} dr.
\end{displaymath}
Let $x_{1}$ and $x_{2}$ be two neighbors in 
$\J^{\bullet}$. If $\widebreve{X}^{\J^{\bullet}}_{q}=x_{1}$
and $\widebreve{\O}^{\J^{\bullet}}_{q}(\{x_{1},x_{2}\})=1$, then:
\begin{itemize}
\item $\widebreve{X}^{\J^{\bullet}}_{q}$ jumps to $x_{2}$ with rate
\begin{displaymath}
\dfrac{1}{2\vert x_{2}-x_{1}\vert}
\dfrac{\breve{\lambda}_{q}^{\J^{\bullet}}
(x_{2})^{\frac{1}{2}}}
{\breve{\lambda}_{q}^{\J^{\bullet}}
(x_{1})^{\frac{1}{2}}}.
\end{displaymath}
\item $\widebreve{\O}^{\J^{\bullet}}_{q}(\{x_{1},x_{2}\})$ is set to $0$ with rate
\begin{displaymath}
\dfrac{1}{\vert x_{2}-x_{1}\vert}
\dfrac{\breve{\lambda}_{q}^{\J^{\bullet}}
(x_{2})^{\frac{1}{2}}}
{\breve{\lambda}_{q}^{\J^{\bullet}}
(x_{1})^{\frac{1}{2}}}
\big(\exp(\vert x_{2}-x_{1}\vert^{-1}
\breve{\lambda}_{q}^{\J^{\bullet}}
(x_{1})^{\frac{1}{2}}
\breve{\lambda}_{q}^{\J^{\bullet}}
(x_{2})^{\frac{1}{2}})-1\big)^{-1}.
\end{displaymath}
$\widebreve{X}^{\J^{\bullet}}_{q\geq 0}$ jumps instantaneously jumps to $x_{2}$ or stays in $x_{1}$ depending on which of the two vertices remains connected to $0$ by open edges in $\widebreve{X}^{\J^{\bullet}}_{q}$.
\end{itemize}
The process $(\widebreve{X}^{\J^{\bullet}}_{q},\breve{\lambda}^{\J^{\bullet}}_{q}, \widebreve{\O}^{\J^{\bullet}}_{q})_{q\geq 0}$
is defined up to time
\begin{displaymath}
\widebreve{\mathcal{Q}}^{\J^{\bullet}}=\sup\lbrace q\geq 0\vert
\breve{\lambda}_{q}^{\J^{\bullet}}(\widebreve{X}^{\J^{\bullet}}_{q})>0\rbrace.
\end{displaymath}
By construction, $\widebreve{X}^{\J^{\bullet}}_{q}$ is always in the same connected component induced by open edges in 
$\widebreve{\O}^{\J^{\bullet}}_{q}$ as the vertex $0$.
$(\widebreve{\O}^{\J^{\bullet}}_{q})_{0\leq q 
\leq\widebreve{\mathcal{Q}}^{\J^{\bullet}}}$ 
is a non-increasing family. It is easy to see that
a.s. $\widebreve{X}^{\J^{\bullet}}_{q}(\widebreve{\mathcal{Q}}^{\J^{\bullet}})=0$ and the edges adjacent to $0$ are closed in
$\widebreve{\O}^{\J^{\bullet}}_{\widebreve{\mathcal{Q}}^{\J^{\bullet}}}$.

Let be $\J^{\ast}=\J^{\bullet}\cap I(\phi^{(a)})$, 
$I(\phi^{(a)})$ being as in Theorem \ref{ThmRK}. 
We consider the process 
\\$(\widecheck{X}_{q}^{\J^{\ast}},
\check{\lambda}_{q}^{\J^{\ast}}(x))
_{x\in \J^{\ast},0\leq q\leq 
\widecheck{\mathcal{Q}}^{\J^{\ast}}}$ 
following the definition
\eqref{EqDSRP1}, \eqref{EqDSRP2} and \eqref{EqDSRP3}, with
$\J=\J^{\ast}$, $\widecheck{X}_{0}^{\J^{\ast}}=0$, and
$\widecheck{X}_{q}^{\J^{\ast}}=\phi^{(a)}(x)^{2}$,
$x\in \J^{\ast}$. By construction, 
$\widebreve{X}^{\J^{\bullet}}_{q}$ takes values in $\J^{\ast}$. One can couple
$(\widebreve{X}^{\J^{\bullet}}_{q},\breve{\lambda}^{\J^{\bullet}}_{q}, \widebreve{\O}^{\J^{\bullet}}_{q})_{0\leq q\leq \widebreve{\mathcal{Q}}^{\J^{\bullet}}}$
and
$(\widecheck{X}_{q}^{\J^{\ast}},
\check{\lambda}_{q}^{\J^{\ast}}(x))
_{x\in \J^{\ast},0\leq q\leq 
\widecheck{\mathcal{Q}}^{\J^{\ast}}}$ 
such that on the event $\widecheck{\mathcal{Q}}^{\J^{\ast}}\neq 0$
(i.e. $\J^{\ast}$ not reduced to $\{ 0\}$),
\begin{displaymath}
\widecheck{\mathcal{Q}}^{\J^{\ast}}=
\sup\lbrace q\geq 0\vert
\widebreve{\O}^{\J^{\bullet}}_{q}=\widebreve{\O}^{\J^{\bullet}}_{0}\rbrace,
\end{displaymath}
and
\begin{displaymath}
\forall q\in[0,\widecheck{\mathcal{Q}}^{\J^{\ast}}], 
\widecheck{X}_{q}^{\J^{\ast}}=
\widebreve{X}^{\J^{\bullet}}_{q}.
\end{displaymath}
$\widecheck{\mathcal{Q}}^{\J^{\ast}}$ is the first time $q$ when one more edge of $\J^{\bullet}$ is closed in $\widebreve{\O}^{\J^{\bullet}}_{q}$.
Note that after time $\widecheck{\mathcal{Q}}^{\J^{\ast}}$, the processes
$\widecheck{X}_{q}^{\J^{\ast}}$ and 
$\widebreve{X}^{\J^{\bullet}}_{q}$ do not coincide anymore.

\begin{prop}[Lupu-Sabot-Tarrès \cite{LST2017InvRK}, Proposition 3.4]
\label{PropLST0}
With the notations above, the process
\begin{displaymath}
(\widebreve{X}^{\J^{\bullet}}_{q},\breve{\lambda}^{\J^{\bullet}}_{q}, \widebreve{\O}^{\J^{\bullet}}_{q})_{0\leq q\leq \widebreve{\mathcal{Q}}^{\J^{\bullet}}}
\end{displaymath}
has the same law as the time-reversed process
\begin{displaymath}
(X^{\J^{\bullet}}_{Q^{\J^{\bullet},\beta}(\tau^{\beta}_{a^{2}/2})-q},\lambda^{\J^{\bullet}}_{Q^{\J^{\bullet},\beta}(\tau^{\beta}_{a^{2}/2})-q},\O^{\J^{\bullet}}_{Q^{\J^{\bullet},\beta}(\tau^{\beta}_{a^{2}/2})-q})_{0\leq q\leq Q^{\J^{\bullet},\beta}(\tau^{\beta}_{a^{2}/2})}.
\end{displaymath}
\end{prop}

In particular, by considering
$(\widebreve{X}^{\J^{\bullet}}_{q},\breve{\lambda}^{\J^{\bullet}}_{q}, \widebreve{\O}^{\J^{\bullet}}_{q})_{0\leq q\leq \widebreve{\mathcal{Q}}^{\J^{\bullet}}}$
up to time $\widecheck{\mathcal{Q}}^{\J^{\ast}}$ we get the following:

\begin{cor}
\label{PropLST1}
Let $\widecheck{T}^{\J^{\ast},\beta,a}$ be $0$ if
$\J^{\ast}$ is reduced to $\lbrace 0\rbrace$,  and otherwise,
\begin{displaymath}
\widecheck{T}^{\J^{\ast},\beta,a}=
\tau^{\beta}_{a^{2}/2}-
\sup\lbrace t\in [0,\tau^{\beta}_{a^{2}/2}]\vert
\beta_{t}\in(\min\J^{\ast},\max\J^{\ast}),
\phi^{(0)}(\beta_{t})=0~\text{and}~
\forall s\in[0,t),\beta_{s}\neq\beta_{t}\rbrace.
\end{displaymath}
Then, the joint law of
\begin{displaymath}
(\J^{\ast}, \phi^{(a)}(x),
X^{\J^{\bullet}}_{Q^{\J^{\bullet},\beta}(\tau^{\beta}_{a^{2}/2})-q})
_{x\in\J^{\ast}, 0\leq q\leq
Q^{\J^{\bullet},\beta}(\tau^{\beta}_{a^{2}/2})-
Q^{\J^{\bullet},\beta}(\tau^{\beta}_{a^{2}/2}-
\widecheck{T}^{\J^{\ast},\beta, a})
}
\end{displaymath}
is the same as the joint law of
\begin{displaymath}
(\J^{\ast}, \phi^{(a)}(x),\widecheck{X}^{\J^{\ast}}_{q})
_{x\in\J^{\ast}, 0\leq q\leq \widecheck{\mathcal{Q}}^{\J^{\ast}}}.
\end{displaymath}
\end{cor}

\begin{proof}
The identity comes from Proposition \ref{PropLST0} and the fact that,
in case $\J^{\ast}$ is not reduced to $\lbrace 0\rbrace$,
\begin{displaymath}
Q^{\J^{\bullet},\beta}(\tau^{\beta}_{a^{2}/2}-
\widecheck{T}^{\J^{\ast},\beta, a})=
\inf\Big\{q\geq 0\vert
\O^{\J^{\bullet}}_{q}=
\O^{\J^{\bullet}}_{Q^{\J^{\bullet},\beta}(\tau^{\beta}_{a^{2}/2})}
\Big\},
\end{displaymath}
and
\begin{displaymath}
Q^{\J^{\bullet},\beta}(\tau^{\beta}_{a^{2}/2})-
Q^{\J^{\bullet},\beta}(\tau^{\beta}_{a^{2}/2}-
\widecheck{T}^{\J^{\ast},\beta, a})=
\sup\Big\{q\in [0,Q^{\J^{\bullet},\beta}(\tau^{\beta}_{a^{2}/2})]\vert
\O^{\J^{\bullet}}_{Q^{\J^{\bullet},\beta}(\tau^{\beta}_{a^{2}/2})-q}=
\O^{\J^{\bullet}}_{Q^{\J^{\bullet},\beta}(\tau^{\beta}_{a^{2}/2})}
\Big\}.
\qedhere
\end{displaymath}
\end{proof}

\section{Convergence for squared GFF initial occupation profile}
\label{SecBBGFF}

We use the notations of the previous section.
First we will check that the condition \eqref{EqCond} is satisfied
by $\check{\lambda}_{0}^{\ast}(x)=\phi^{(a)}(x)^{2}$.

\begin{lemma}
A.s. we have that
\begin{displaymath}
\int_{\inf I(\phi^{(a)})}\phi^{(a)}(x)^{-2} dx = +\infty,
\qquad
\int^{\sup I(\phi^{(a)})}\phi^{(a)}(x)^{-2} dx = +\infty.
\end{displaymath}
\end{lemma}

\begin{proof}
Conditional on $\inf I(\phi^{(a)})$, 
$(\phi^{(a)}(\inf I(\phi^{(a)})+x)/\sqrt{2})
_{0\leq x\leq \vert\inf I(\phi^{(a)})\vert/2}$
is absolutely continuous with respect to a Bessel $3$ process 
starting from $0$. 
So we only need to check that given $(\rho(x))_{x\geq 0}$
a Bessel 3 process starting from $0$, 
\begin{displaymath}
\int_{0}\rho(x)^{-2} dx = +\infty.
\end{displaymath}
For $h>0$, let $\chi_{h}^{\rho}$ denote the first ``time" $x$ when
$\rho(x)$ reaches the level $h$. Then,
\begin{displaymath}
\int_{0}^{\chi_{1}^{\rho}}\rho(x)^{-2} dx=
\sum_{k\geq 0}
\int_{\chi_{2^{-k-1}}^{\rho}}^{\chi_{2^{-k}}^{\rho}}\rho(x)^{-2} dx.
\end{displaymath}
By the strong Markov property of $\rho$, the sum on the right-hand side is a sum of positive independent terms. Moreover, by Brownian scaling satisfied by $\rho$, these terms are identically distributed. So the sum is a.s. infinite.
\end{proof}

Now we consider $\J^{\bullet}=\Z_{n}=2^{-n}\Z$, and
$\J^{\ast}=\Z^{\ast}_{n}=\Z_{n}\cap I(\phi^{(a)})$.
Let
\begin{displaymath}
X^{(n)}_{t}=\beta_{(Q^{\Z_{n},\beta})^{-1}(2^{n}t)}.
\end{displaymath}

\begin{lemma}
\label{LemEmb}
The process 
\begin{multline}
\label{E3}
(X^{(n)}
_{(2^{-n}Q^{\Z_{n},\beta}(\tau^{\beta}_{a^{2}/2})-t)},
\phi^{(a)}(x)^{2}
-2\ell^{\beta}_{\tau^{\beta}_{a^{2}/2}}(x)
+2\ell^{\beta}
_{(Q^{\Z_{n},\beta})^{-1}(Q^{\Z_{n},\beta}(\tau^{\beta}_{a^{2}/2})
-2^{n}t)}(x))
\\
x\in\Z^{\ast}_{n},
0\leq t\leq (2^{-n}Q^{\Z_{n},\beta}(\tau^{\beta}_{a^{2}/2})-
2^{-n}Q^{\Z_{n},\beta}(\tau^{\beta}_{a^{2}/2}-
\widecheck{T}^{\Z^{\ast}_{n},\beta,a})),
\end{multline}
interpolated linearly outside $x\in\Z^{\ast}_{n}$,
converges a.s. in the uniform topology to
\begin{equation}
\label{E4}
(\beta_{\tau^{\beta}_{a^{2}/2}-t},
\phi^{(a)}(x)^{2}
-2\ell^{\beta}_{\tau^{\beta}_{a^{2}/2}}(x)
+2\ell^{\beta}_{\tau^{\beta}_{a^{2}/2}-t}(x))_{
x\in I(\phi^{(a)}),0\leq t\leq \widecheck{T}^{\beta,a}}
\end{equation}
as $n\to +\infty$.
\end{lemma}

\begin{proof}
One needs to show that, on one hand, as $n\to +\infty$, a.s.
$(Q^{\Z_{n},\beta})^{-1}(2^{n}t)$ converges to $t$ uniformly
on $[0,\tau^{\beta}_{a^{2}/2}]$, and on the other hand
$\widecheck{T}^{\Z^{\ast}_{n},\beta,a}$ converges a.s. to
$\widecheck{T}^{\beta,a}$.

The first convergence comes from the fact that
\begin{displaymath}
2^{-n}\sum_{x\in 2^{-n}\Z}\ell^{\beta}_{t}(x)
\end{displaymath}
converges to
\begin{displaymath}
t=\int_{\R}\ell^{\beta}_{t}(x) dx
\end{displaymath}
uniformly on compact intervals of time.

The second convergence comes from the fact that
$(\min \Z^{\ast}_{n}, \max \Z^{\ast}_{n})$ is a non-decreasing sequence of intervals converging to $I(\phi^{(a)})$, and thus, a.s., for $n$ large enough, 
$\widecheck{T}^{\beta,a}=\widecheck{T}^{\Z^{\ast}_{n},\beta,a}$.
\end{proof}

Let $\widecheck{X}^{\ast(n)}_{t}=\widecheck{X}^{\Z^{\ast}_{n}}_{2^{n}t}$ be the self-repelling jump process on $\Z^{\ast}_{n}$, accelerated by the factor $2^{n}$. It is the same process as in Theorem \ref{ThmSIConv}, but with a random initial occupation profile
\begin{displaymath}
\check{\lambda}^{\ast(n)}_{0}(x)=
\check{\lambda}^{\Z^{\ast}_{n}}_{0}(x)=
\phi^{(a)}(x)^{2}, x\in\Z^{\ast}_{n}.
\end{displaymath}
We will show that $\widecheck{X}^{\ast(n)}_{t}$ converges in law as 
$n\to +\infty$ to our self-repelling diffusion. For this we will use a method that appears in \cite{LST2018LRM}, and construct a discrete analogue of the divergent Bass-Burdzy flow.

\begin{prop}
\label{PropDiscrBB}
Given $\phi^{(a)}$, let $\widecheck{X}^{\ast(n)}_{t}$ be the process on 
$\Z^{\ast}_{n}$ defined above, and
\begin{displaymath}
\check{\lambda}^{\ast(n)}_{t}(x)=
\phi^{(a)}(x)^{2}-2^{n+1}\int_{0}^{t} 
\1_{\widecheck{X}^{\ast (n)}_{s}=x} ds,~~
x\in\Z^{\ast}_{n}.
\end{displaymath}
Then, as $n\to +\infty$, the process
\begin{equation}
\label{EqDiscr2}
(\widecheck{X}^{\ast(n)}
_{t\wedge 2^{-n}\widecheck{\mathcal{Q}}^{\Z^{\ast}_{n}}},
\check{\lambda}^{\ast(n)}
_{t\wedge 2^{-n}\widecheck{\mathcal{Q}}^{\Z^{\ast}_{n}}}(x))
_{x\in\Z^{\ast}_{n},t\geq 0},
\end{equation}
interpolated linearly outside $x\in\Z^{\ast}_{n}$, converges in law to the self repelling diffusion
\begin{displaymath}
(\widecheck{X}^{\ast}_{t\wedge\widecheck{T}^{\ast}},
\check{\lambda}^{\ast}_{t\wedge\widecheck{T}^{\ast}}(x))
_{x\in I(\phi^{(a)}),t\geq 0}
\end{displaymath}
with $\widecheck{X}^{\ast}_{0}=0$ and the initial occupation profile
$\check{\lambda}^{\ast}_{0}(x)=\phi^{(a)}(x)^{2}$.
\end{prop}

Before proceeding to the proof of Proposition \ref{PropDiscrBB},
let us explain how it implies Theorem \ref{ThmRKInv}.

\begin{proof}[Proof of Theorem \ref{ThmRKInv}]
On one hand, according to Proposition \ref{PropDiscrBB},
\begin{equation}
\label{E1}
(\widecheck{X}^{\ast(n)}_{t},
\check{\lambda}^{\ast(n)}_{t}(x))
_{x\in\Z^{\ast}_{n},0\leq t
\leq 2^{-n}\widecheck{\mathcal{Q}}^{\Z^{\ast}_{n}}}
\end{equation}
converges in law to
\begin{equation}
\label{E2}
(\widecheck{X}^{\ast}_{t},
\check{\lambda}^{\ast}_{t}(x))
_{x\in I(\phi^{(a)}),0\leq t\leq \widecheck{T}^{\ast}} .
\end{equation}
On the other hand, according to Corollary \ref{PropLST1},
\eqref{E1} has the same distribution as \eqref{E3}.
According to Lemma \ref{LemEmb},
\eqref{E3} in turn converges a.s. to \eqref{E4}.
This means that \eqref{E2} has the same distribution as \eqref{E4},
which is exactly what we want.
\end{proof}

\begin{proof}[Proof of Proposition \ref{PropDiscrBB}]
From Corollary \ref{PropLST1} and Lemma \ref{LemEmb} we already know that the process \eqref{EqDiscr2} has a limit in law, but we want another description of the limit, which we will obtain by convergence.
We will need the triple 
\begin{displaymath}
(\phi^{(0)}(x)^{2},\beta_{t},
\phi^{(a)}(x)^{2})_{x\in\R, 0\leq t\leq \tau^{\beta}_{a^{2}/2}},
\end{displaymath}
jointly distributed as in the Ray-Knight coupling 
(Definition \ref{Def RK coupl}).
We will also assume that all of the $\widecheck{X}^{\ast(n)}_{t}$
are defined on the same probability spaces, embedded in
$\beta_{t}$ as in Corollary \ref{PropLST1}.

We introduce 
$(\widecheck{S}_{t}^{\ast(n)})
_{0\leq t\leq 2^{-n}\widecheck{\mathcal{Q}}^{\Z_{n}^{\ast}}}$ 
a family of maps
$\R\rightarrow\R$, parametrized by $t$. For a given $n$, the family
is characterized by the following:
\begin{itemize}
\item For all $x$ such that 
$x$ and $x+2^{-n}$ are in $\Z^{\ast}_{n}$, and for all
$t\in[0,2^{-n}\widecheck{\mathcal{Q}}^{\Z_{n}^{\ast}}]$,
\begin{displaymath}
\widecheck{S}_{t}^{\ast(n)}(x+2^{-n})-
\widecheck{S}_{t}^{\ast(n)}(x)=
2^{-n}
\check{\lambda}^{\ast (n)}_{t}(x)^{-\frac{1}{2}}
\check{\lambda}^{\ast (n)}_{t}(x+2^{-n})^{-\frac{1}{2}}.
\end{displaymath}
\item $\widecheck{S}_{0}^{\ast(n)}(0)=0$.
\item For every $x\in\Z^{\ast}_{n}$,
$t\mapsto\widecheck{S}_{t}^{\ast(n)}(x)$ is constant on each time interval on which
$\widecheck{X}^{\ast (n)}_{t}=x$.
\item For each $x\in\Z^{\ast}_{n}$,
$t\mapsto \widecheck{S}_{t}^{\ast(n)}(x)$ is continuous.
\item For each $t$, $\widecheck{S}_{t}^{\ast(n)}$ is interpolated linearly between points of $\Z^{\ast}_{n}$.
\item Below $\min \Z^{\ast}_{n}$ and above
$\max \Z^{\ast}_{n}$,
$x\mapsto \widecheck{S}_{t}^{\ast(n)}(x)$ has constant slope $1$.
\end{itemize}
By construction, $x\mapsto \widecheck{S}_{t}^{\ast(n)}$ is continuous strictly increasing.
We see $\widecheck{S}_{t}^{\ast(n)}$ as a time-dependent change of scale.
It has been constructed in such a way that
the process
$(\widecheck{S}_{t}^{\ast(n)}(\widecheck{X}^{\ast (n)}_{t}))_{t}$
is a local martingale;
see Lemma \ref{LemClaim2} further below.

For $x\in I(\phi^{(a)})$ and
$t\in [0,\tau^{\beta}_{a^{2}/2})$, set
\begin{displaymath}
\bar{S}^{\ast}_{t}(x)=\int^{x}_{0}
(\phi^{(a)}(r)^{2}-2\ell^{\beta}_{\tau^{\beta}_{a^{2}/2}}
+2\ell^{\beta}_{\tau^{\beta}_{a^{2}/2}-t})^{-1} dr.
\end{displaymath}
$x\mapsto \bar{S}^{\ast}_{t}(x)$ is an increasing diffeomorphism from $I(\phi^{(a)})$ to $\R$. Clearly, we have the following

\begin{lemma}
\label{LemClaim1}

A.s.
$\widecheck{S}_{t}^{\ast(n)}(x)-\widecheck{S}_{t}^{\ast(n)}(0)$
converges to
$\bar{S}^{\ast}_{t}(x)-\bar{S}^{\ast}_{t}(0)$
uniformly for $(x,t)$ in compact subsets of
$I(\phi^{(a)})\times [0,\tau^{\beta}_{a^{2}/2})$.
Similarly, a.s.
$(y,t)\mapsto (\widecheck{S}_{t}^{\ast(n)})^{-1}
(y+\widecheck{S}_{t}^{\ast(n)}(0))$ converges to
$(y,t)\mapsto (\bar{S}_{t}^{\ast})^{-1}
(y+\bar{S}_{t}^{\ast}(0))$
uniformly on compact subsets of
$\R\times [0,\tau^{\beta}_{a^{2}/2})$.
\end{lemma}

Let be 
$M^{(n)}_{t}=\widecheck{S}_{t}^{\ast(n)}(\widecheck{X}^{\ast (n)}_{t})$.
Let $t^{\ast(n)}_{\partial\Z^{\ast}_{n}}$ be the first
time $\widecheck{X}^{\ast (n)}_{t}$ hits
$\min \Z^{\ast}_{n}$ or $\max \Z^{\ast}_{n}$.
We introduce the stopping time 
$t^{\ast(n)}_{\partial\Z^{\ast}_{n}}$
to avoid considering what happens after
$\widecheck{X}^{\ast (n)}_{t}$ hits the boundary of
the domain $\Z^{\ast}_{n}$.

\bigskip

\begin{lemma} 
\label{LemClaim2}

The process
$(M^{(n)}_{t\wedge 
t^{\ast(n)}_{\partial\Z^{\ast}_{n}}
\wedge 2^{-n}\widecheck{\mathcal{Q}}^{\Z_{n}^{\ast}}})
_{t\geq 0}$ is a local martingale in the filtration of
\\$(\phi^{(a)},\widecheck{X}^{\ast (n)}
_{t\wedge 2^{-n}\widecheck{\mathcal{Q}}^{\Z_{n}^{\ast}}},
t^{\ast(n)}_{\partial\Z^{\ast}_{n}}
\1_{t^{\ast(n)}_{\partial\Z^{\ast}_{n}}\leq t}
)$.
\end{lemma}

\begin{proof}
Indeed, consider the following stopping times for the above filtration:
$t^{\ast(n)}_{k}$ the first time $\widecheck{X}^{\ast (n)}_{t}$ performs $k$ jumps, and
\begin{equation}
\label{EqTneps}
\widecheck{T}^{\ast (n)}_{\varepsilon}=
\sup\lbrace t\geq \vert 
\check{\lambda}^{\ast (n)}_{t}(\widecheck{X}^{\ast (n)}_{t})
>\varepsilon
\rbrace.
\end{equation}
Then 
$\vert M^{(n)}_{t\wedge t^{\ast(n)}_{k}\wedge
\widecheck{T}^{\ast (n)}_{\varepsilon}\wedge 
t^{\ast(n)}_{\partial\Z^{\ast}_{n}}
\wedge 2^{-n}\widecheck{\mathcal{Q}}^{\Z_{n}^{\ast}}}
\vert$
is bounded by
\begin{displaymath}
k 2^{-n}(\min_{\Z^{\ast}_{n}}\check{\lambda}^{\ast (n)}_{0}\wedge\varepsilon)^{-1}.
\end{displaymath}
Moreover,
\begin{displaymath}
\sup_{k\in\mathbb{N}, \varepsilon >0}
t^{\ast(n)}_{k}\wedge
\widecheck{T}^{\ast (n)}_{\varepsilon}\wedge 
t^{\ast(n)}_{\partial\Z^{\ast}_{n}}
\wedge 2^{-n}\widecheck{\mathcal{Q}}^{\Z_{n}^{\ast}}
=
t^{\ast(n)}_{\partial\Z^{\ast}_{n}}
\wedge 2^{-n}\widecheck{\mathcal{Q}}^{\Z_{n}^{\ast}}
~~\text{a.s.}
\end{displaymath}
To see that
$(M^{(n)}_{t\wedge t^{\ast(n)}_{k}\wedge
\widecheck{T}^{\ast (n)}_{\varepsilon}\wedge 
t^{\ast(n)}_{\partial\Z^{\ast}_{n}}
\wedge 2^{-n}\widecheck{\mathcal{Q}}^{\Z_{n}^{\ast}}})
_{t\geq 0}$
is a martingale, observe that at time $t$, if
$\widecheck{X}^{\ast (n)}_{t}=x\in \Z^{\ast}_{n}\setminus
\lbrace \min\Z^{\ast}_{n}, \max \Z^{\ast}_{n} \rbrace$,
$\widecheck{X}^{\ast (n)}_{t}$ jumps left with rate
\begin{displaymath}
2^{2n-1}\dfrac{\check{\lambda}_{t}^{\ast(n)}(x-2^{-n})^{\frac{1}{2}}}
{\check{\lambda}_{t}^{\ast(n)}(x)^{\frac{1}{2}}},
\end{displaymath}
and then $M^{(n)}_{t}$ decreases by
\begin{displaymath}
2^{-n}\check{\lambda}_{t}^{\ast(n)}(x-2^{-n})^{-\frac{1}{2}}
\check{\lambda}_{t}^{\ast(n)}(x)^{-\frac{1}{2}},
\end{displaymath}
and $\widecheck{X}^{\ast (n)}_{t}=x\in \Z^{\ast}_{n}$ jumps right with rate
\begin{displaymath}
2^{2n-1}\dfrac{\check{\lambda}_{t}^{\ast(n)}(x+2^{-n})^{\frac{1}{2}}}
{\check{\lambda}_{t}^{\ast(n)}(x)^{\frac{1}{2}}},
\end{displaymath}
and then $M^{(n)}_{t}$ increases by
\begin{displaymath}
2^{-n}\check{\lambda}_{t}^{\ast(n)}(x+2^{-n})^{-\frac{1}{2}}
\check{\lambda}_{t}^{\ast(n)}(x)^{-\frac{1}{2}},
\end{displaymath}
so the average variation of $M^{(n)}_{t}$ is $0$.
\end{proof}

Next we will apply a time-change which will make 
$(M^{(n)}_{t\wedge 
t^{\ast(n)}_{\partial\Z^{\ast}_{n}}
\wedge 2^{-n}\widecheck{\mathcal{Q}}^{\Z_{n}^{\ast}}})
_{t\geq 0}$
into a martingale with normalized variance. Let be
\begin{displaymath}
U^{(n)}(t)=\int_{0}^{t}\dfrac{1}{2}
\check{\lambda}_{s}^{\ast(n)}(\widecheck{X}^{\ast (n)}_{s})^{-\frac{3}{2}}
\Big(
\check{\lambda}_{s}^{\ast(n)}(\widecheck{X}^{\ast (n)}_{s}-2^{-n})
^{-\frac{1}{2}}+
\check{\lambda}_{s}^{\ast(n)}(\widecheck{X}^{\ast (n)}_{s}+2^{-n})
^{-\frac{1}{2}}
\Big) ds.
\end{displaymath}
Let 
$\widecheck{\mathcal{U}}^{(n)}=
U^{(n)}(t^{\ast(n)}_{\partial\Z^{\ast}_{n}}
\wedge 2^{-n}\widecheck{\mathcal{Q}}^{\Z_{n}^{\ast}})$.
By considering the rate of jumps and the size of jumps of
$M^{(n)}_{t}$, we immediately get the following:

\begin{lemma}
\label{LemClaim2bis}
The process
$((M^{(n)}_{t\wedge 
t^{\ast(n)}_{\partial\Z^{\ast}_{n}}
\wedge 2^{-n}\widecheck{\mathcal{Q}}^{\Z_{n}^{\ast}}})^{2}
-U^{(n)}(t)\wedge\widecheck{\mathcal{U}}^{(n)})
_{t\geq 0}$ is a local martingale in the filtration of
$(\phi^{(a)},\widecheck{X}^{\ast (n)}
_{t\wedge 2^{-n}\widecheck{\mathcal{Q}}^{\Z_{n}^{\ast}}},
t^{\ast(n)}_{\partial\Z^{\ast}_{n}}
\1_{t^{\ast(n)}_{\partial\Z^{\ast}_{n}}\leq t}
)$.
\end{lemma}

Let be
\begin{displaymath}
Z^{(n)}_{u}=M^{(n)}_{(U^{(n)})^{-1}(u)}.
\end{displaymath}

\bigskip

\begin{lemma}
\label{LemClaim3}

$(Z^{(n)}_{u\wedge \widecheck{\mathcal{U}}^{(n)}})_{u\geq 0}$ 
is a martingale
in the filtration of
$(\phi^{(a)},Z^{(n)}_{u}, 
\widecheck{\mathcal{U}}^{(n)}
\1_{\widecheck{\mathcal{U}}^{(n)}\leq u}
)$. Moreover, for any
$0\leq u_{1}<u_{2}$,
\begin{multline}
\label{EqVarMart}
\E[(Z^{(n)}_{u_{2}\wedge \widecheck{\mathcal{U}}^{(n)}}
-Z^{(n)}_{u_{1}\wedge \widecheck{\mathcal{U}}^{(n)}})^{2}
\vert \phi^{(a)}, (Z^{(n)}_{u})_{0\leq u\leq u_{1}},
\widecheck{\mathcal{U}}^{(n)}\1_{\widecheck{\mathcal{U}}^{(n)}\leq u_{1}}]=\\
\E[u_{2}\wedge 
\widecheck{\mathcal{U}}^{(n)}
-u_{1}\wedge \widecheck{\mathcal{U}}^{(n)}
\vert \phi^{(a)}, (Z^{(n)}_{u})_{0\leq u\leq u_{1}},
\widecheck{\mathcal{U}}^{(n)}\1_{\widecheck{\mathcal{U}}^{(n)}\leq u_{1}}],
\end{multline}
or equivalently,
$((Z^{(n)}_{u\wedge \widecheck{\mathcal{U}}^{(n)}})^{2}
-u\wedge\widecheck{\mathcal{U}}^{(n)})_{u\geq 0}$ 
is a martingale in the filtration of
$(\phi^{(a)},Z^{(n)}_{u}, 
\widecheck{\mathcal{U}}^{(n)}
\1_{\widecheck{\mathcal{U}}^{(n)}\leq u}
)$.
\end{lemma}

\begin{proof}
First not that, since
$(M^{(n)}_{t\wedge t^{\ast(n)}_{k} \wedge
\widecheck{T}^{\ast (n)}_{\varepsilon}\wedge 
t^{\ast(n)}_{\partial\Z^{\ast}_{n}}
\wedge 2^{-n}\widecheck{\mathcal{Q}}^{\Z_{n}^{\ast}}})
_{t\geq 0}$ is a bounded martingale, so is
$(Z^{(n)}_{u\wedge 
U^{(n)}(t^{\ast(n)}_{k}\wedge
\widecheck{T}^{\ast (n)}_{\varepsilon})
\wedge\widecheck{\mathcal{U}}^{(n)}})_{u\geq 0}$. Moreover,
with the sizes of jumps and the jump rates, one sees that
$dU^{(n)}_{t}$ is the average squared variation of
$M^{(n)}_{t}$ during $dt$. 
So after the time change, for
$Z^{(n)}_{u}$,
\begin{multline*}
\E[(Z^{(n)}_{u_{2}
\wedge 
U^{(n)}(t^{\ast(n)}_{k}\wedge
\widecheck{T}^{\ast (n)}_{\varepsilon})
\wedge \widecheck{\mathcal{U}}^{(n)}}
-Z^{(n)}_{u_{1}
\wedge U^{(n)}(t^{\ast(n)}_{k}\wedge
\widecheck{T}^{\ast (n)}_{\varepsilon})
\wedge \widecheck{\mathcal{U}}^{(n)}})^{2}
\vert \phi^{(a)}, (Z^{(n)}_{u})_{0\leq u\leq u_{1}},
\widecheck{\mathcal{U}}^{(n)}\1_{\widecheck{\mathcal{U}}^{(n)}\leq u_{1}}]=\\
\E[u_{2}\wedge U^{(n)}(t^{\ast(n)}_{k}\wedge
\widecheck{T}^{\ast (n)}_{\varepsilon})
\wedge 
\widecheck{\mathcal{U}}^{(n)}
-u_{1}\wedge U^{(n)}(t^{\ast(n)}_{k}\wedge
\widecheck{T}^{\ast (n)}_{\varepsilon})
\wedge \widecheck{\mathcal{U}}^{(n)}
\vert \phi^{(a)}, (Z^{(n)}_{u})_{0\leq u\leq u_{1}},
\widecheck{\mathcal{U}}^{(n)}\1_{\widecheck{\mathcal{U}}^{(n)}\leq u_{1}}].
\end{multline*}
For a fixed $u\geq 0$,
$(Z^{(n)}_{u\wedge 
U^{(n)}(t^{\ast(n)}_{k}\wedge
\widecheck{T}^{\ast (n)}_{k^{-1}})
\wedge\widecheck{\mathcal{U}}^{(n)}})_{k\geq 1}$
is a martingale parametrized by $k\in\N^{\ast}$. It converges a.s.
to $Z^{(n)}_{u\wedge\widecheck{\mathcal{U}}^{(n)}}$ and is bounded in
$\mathbb{L}^{2}$, so the convergence is also in
$\mathbb{L}^{2}$. It follows that
$(Z^{(n)}_{u\wedge \widecheck{\mathcal{U}}^{(n)}})_{u\geq 0}$ is a martingale and \eqref{EqVarMart}.
\end{proof}

For $\varepsilon>0$ and $n\in\N^{\ast}$, 
we consider $\widecheck{T}^{\ast (n)}_{\varepsilon}$ the time defined
by \eqref{EqTneps}. Let be
$(\widetilde{Z}^{(n,\varepsilon)}_{u})_{u\geq 0}$ the process, which
up to time 
$U^{(n)}(\widecheck{T}^{\ast (n)}_{\varepsilon})\wedge
\widecheck{\mathcal{U}}^{(n)}$
coincides with $Z^{(n)}_{u}$, and after that time continues as a standard Brownian motion starting from
$Z^{(n)}_{U^{(n)}(\widecheck{T}^{\ast (n)}_{\varepsilon})\wedge
\widecheck{\mathcal{U}}^{(n)}}$, conditional of that value independent of everything else.

\begin{lemma}
\label{LemClaim4} 

As $n\to +\infty$, the pair
$(\phi^{(a)},\widetilde{Z}^{(n,\varepsilon)}_{u})_{u\geq 0}$ converges in law, for the uniform convergence on compact subsets, to 
$(\phi^{(a)},B_{u})_{u\geq 0}$, where $(B_{u})_{u\geq 0}$ is
a standard Brownian motion starting from $0$, 
independent of $\phi^{(a)}$. 
\end{lemma}

\begin{proof}
The convergence of 
$(\widetilde{Z}^{(n,\varepsilon)}_{u})_{u\geq 0}$ to
$(B_{u})_{u\geq 0}$ follows from Theorem 1.4,
Section 7.1 in \cite{EthierKurtz1986Markov}. 
To apply it, we use the following:
\begin{itemize}
\item $(\widetilde{Z}^{(n,\varepsilon)}_{u})_{u\geq 0}$ is a martingale.
\item $((\widetilde{Z}^{(n,\varepsilon)}_{u})^{2}-u)_{u\geq 0}$ 
is a martingale by Lemma \ref{LemClaim3}.
\item The jumps of $(\widetilde{Z}^{(n,\varepsilon)}_{u})_{u}$ are bounded by $2^{-n}(\min_{\Z^{\ast}_{n}}\check{\lambda}^{\ast (n)}_{0}\wedge\varepsilon)^{-1}$, and in particular
\begin{displaymath}
\lim_{n\to +\infty}
\E\Big[\max_{u\geq 0}(\widetilde{Z}^{(n,\varepsilon)}_{u}-
\widetilde{Z}^{(n,\varepsilon)}_{u^{-}})^{2}\Big]=0.
\end{displaymath}
\end{itemize}
The independence of $(B_{u})_{u\geq 0}$ from
$\phi^{(a)}$ follows from the fact that the above listed three conditions hold after conditioning by $\phi^{(a)}$.
\end{proof}

We stress that in Lemma \ref{LemClaim4} we neither require
$(B_{u})_{u\geq 0}$ to be defined on the same probability space as the
$\widecheck{X}^{\ast (n)}_{t}$ and
$(\phi^{(0)}(x)^{2},\beta_{t},
\phi^{(a)}(x)^{2})_{x\in\R, 0\leq t\leq \tau^{\beta}_{a^{2}/2}}$,
nor the convergence to be in probability.

Let be, for $t\in[0,\widecheck{T}^{\beta,a}]$,
\begin{displaymath}
U(t)=\int_{0}^{t}
(\phi^{(a)}(\beta_{\tau^{\beta}_{a^{2}/2}-s})^{2}-
2\ell^{\beta}_{\tau^{\beta}_{a^{2}/2}}
(\beta_{\tau^{\beta}_{a^{2}/2}-s})
+2\ell^{\beta}_{\tau^{\beta}_{a^{2}/2}-s}
(\beta_{\tau^{\beta}_{a^{2}/2}-s}))^{-2} ds,
\end{displaymath}
and
\begin{displaymath}
\widecheck{T}^{\beta,a}_{\varepsilon}=
\sup\lbrace t\geq 0\vert
\phi^{(a)}(\beta_{\tau^{\beta}_{a^{2}/2}-s})^{2}-
2\ell^{\beta}_{\tau^{\beta}_{a^{2}/2}}
(\beta_{\tau^{\beta}_{a^{2}/2}-s})
+2\ell^{\beta}_{\tau^{\beta}_{a^{2}/2}-s}
(\beta_{\tau^{\beta}_{a^{2}/2}-s})>\varepsilon
\rbrace.
\end{displaymath}
Clearly, we have the following:

\begin{lemma}
\label{LemClaim5}

For all $\varepsilon>0$, a.s., 
$\widecheck{T}^{\ast (n)}_{\varepsilon}$ converges to
$\widecheck{T}^{\beta,a}_{\varepsilon}$,
$U^{(n)}(t)\wedge U^{(n)}(\widecheck{T}^{\ast (n)}_{\varepsilon})
\wedge \widecheck{\mathcal{U}}^{(n)}$ converges to
$U(t)\wedge U(\widecheck{T}^{\beta,a}_{\varepsilon})$
uniformly on $[0,+\infty)$, and
$(U^{(n)})^{-1}(u)\wedge 
\widecheck{T}^{\ast (n)}_{\varepsilon}\wedge 
t^{\ast(n)}_{\partial\Z^{\ast}_{n}}
\wedge 2^{-n}\widecheck{\mathcal{Q}}^{\Z_{n}^{\ast}}$
converges to
$U^{-1}(u)\wedge \widecheck{T}^{\beta,a}_{\varepsilon}$
uniformly on $[0,+\infty)$.
\end{lemma}

Next, for $u\in [0,\widecheck{\mathcal{U}}^{(n)})$, we define
\begin{displaymath}
\widecheck{\Psi}^{(n)}_{u}(y)= 
\widecheck{S}_{(U^{(n)})^{-1}(u)}^{\ast (n)}\circ
(\widecheck{S}_{0}^{\ast (n)})^{-1}(y),~~y\in\R.
\end{displaymath}
By simple computation, we have the following:

\begin{lemma}
\label{LemClaim6}

For $u\in [0,\widecheck{\mathcal{U}}^{(n)})$
such that $Z^{(n)}_{u}=Z^{(n)}_{u^{-}}$, we have the following expressions and bounds for $\dfrac{\partial}{\partial u}\widecheck{\Psi}^{(n)}_{u}(y)$: 
\begin{itemize}
\item if $\widecheck{\Psi}^{(n)}_{u}(y)=Z^{(n)}_{u}$,
$\dfrac{\partial}{\partial u}\widecheck{\Psi}^{(n)}_{u}(y)=0$;
\item if $\widecheck{\Psi}^{(n)}_{u}(y)\in (Z^{(n)}_{u},
\widecheck{S}^{\ast (n)}_{0}
(\widecheck{X}^{\ast (n)}_{(U^{(n)})^{-1}(u)}+2^{-n}))$,
$$0<\dfrac{\partial}{\partial u}\widecheck{\Psi}^{(n)}_{u}(y)<
\dfrac{
2\check{\lambda}^{\ast(n)}_{(U^{(n)})^{-1}(u)}
(\widecheck{X}^{\ast (n)}_{(U^{(n)})^{-1}(u)}+2^{-n})^{-\frac{1}{2}}}
{\check{\lambda}_{(U^{(n)})^{-1}(u)}^{\ast(n)}
(\widecheck{X}^{\ast (n)}_{(U^{(n)})^{-1}(u)}-2^{-n})
^{-\frac{1}{2}}+
\check{\lambda}_{(U^{(n)})^{-1}(u)}^{\ast(n)}
(\widecheck{X}^{\ast (n)}_{(U^{(n)})^{-1}(u)}+2^{-n})
^{-\frac{1}{2}}}
;$$
\item if $\widecheck{\Psi}^{(n)}_{u}(y)\geq
\widecheck{S}^{\ast (n)}_{0}
(\widecheck{X}^{\ast (n)}_{(U^{(n)})^{-1}(u)}+2^{-n})$,
$$\dfrac{\partial}{\partial u}\widecheck{\Psi}^{(n)}_{u}(y)=
\dfrac{
2\check{\lambda}^{\ast(n)}_{(U^{(n)})^{-1}(u)}
(\widecheck{X}^{\ast (n)}_{(U^{(n)})^{-1}(u)}+2^{-n})^{-\frac{1}{2}}}
{\check{\lambda}_{(U^{(n)})^{-1}(u)}^{\ast(n)}
(\widecheck{X}^{\ast (n)}_{(U^{(n)})^{-1}(u)}-2^{-n})
^{-\frac{1}{2}}+
\check{\lambda}_{(U^{(n)})^{-1}(u)}^{\ast(n)}
(\widecheck{X}^{\ast (n)}_{(U^{(n)})^{-1}(u)}+2^{-n})
^{-\frac{1}{2}}}
;$$
\item if
$\widecheck{\Psi}^{(n)}_{u}(y)\in (\widecheck{S}^{\ast (n)}_{0}
(\widecheck{X}^{\ast (n)}_{(U^{(n)})^{-1}(u)}-2^{-n}),
Z^{(n)}_{u})$,
$$0>\dfrac{\partial}{\partial u}\widecheck{\Psi}^{(n)}_{u}(y)>
\dfrac{-
2\check{\lambda}^{\ast(n)}_{(U^{(n)})^{-1}(u)}
(\widecheck{X}^{\ast (n)}_{(U^{(n)})^{-1}(u)}-2^{-n})^{-\frac{1}{2}}}
{\check{\lambda}_{(U^{(n)})^{-1}(u)}^{\ast(n)}
(\widecheck{X}^{\ast (n)}_{(U^{(n)})^{-1}(u)}-2^{-n})
^{-\frac{1}{2}}+
\check{\lambda}_{(U^{(n)})^{-1}(u)}^{\ast(n)}
(\widecheck{X}^{\ast (n)}_{(U^{(n)})^{-1}(u)}+2^{-n})
^{-\frac{1}{2}}}
;$$
\item if
$\widecheck{\Psi}^{(n)}_{u}(y)\leq\widecheck{S}^{\ast (n)}_{0}
(\widecheck{X}^{\ast (n)}_{(U^{(n)})^{-1}(u)}-2^{-n})$,
$$\dfrac{\partial}{\partial u}\widecheck{\Psi}^{(n)}_{u}(y)=
\dfrac{-
2\check{\lambda}^{\ast(n)}_{(U^{(n)})^{-1}(u)}
(\widecheck{X}^{\ast (n)}_{(U^{(n)})^{-1}(u)}-2^{-n})^{-\frac{1}{2}}}
{\check{\lambda}_{(U^{(n)})^{-1}(u)}^{\ast(n)}
(\widecheck{X}^{\ast (n)}_{(U^{(n)})^{-1}(u)}-2^{-n})
^{-\frac{1}{2}}+
\check{\lambda}_{(U^{(n)})^{-1}(u)}^{\ast(n)}
(\widecheck{X}^{\ast (n)}_{(U^{(n)})^{-1}(u)}+2^{-n})
^{-\frac{1}{2}}}.$$
\end{itemize}
\end{lemma}

\bigskip

For $\varepsilon>0$, let
$\widetilde{\Psi}^{(n,\varepsilon)}_{u}(y)$ be defined as follows.
For $u\in [0,U^{(n)}(\widecheck{T}^{\ast (n)}_{\varepsilon})\wedge
\widecheck{\mathcal{U}}^{(n)}]$,
$\widetilde{\Psi}^{(n,\varepsilon)}_{u}(y)=
\widecheck{\Psi}^{(n)}_{u}(y)$. For
$u>U^{(n)}(\widecheck{T}^{\ast (n)}_{\varepsilon})\wedge
\widecheck{\mathcal{U}}^{(n)}$, 
$\widetilde{\Psi}^{(n,\varepsilon)}_{u}(y)$ is a divergent
Bass-Burdzy flow driven by
$\widetilde{Z}^{(n,\varepsilon)}_{u}$ (which is then a Brownian motion)
satisfying
\begin{displaymath}
\widetilde{\Psi}^{(n,\varepsilon)}_{u}(y)-
\widetilde{\Psi}^{(n,\varepsilon)}_{U^{(n)}
(\widecheck{T}^{\ast (n)}_{\varepsilon})\wedge
\widecheck{\mathcal{U}}^{(n)}}(y)=
\int_{U^{(n)}
(\widecheck{T}^{\ast (n)}_{\varepsilon})\wedge
\widecheck{\mathcal{U}}^{(n)}}^{u}
(\1_{\widetilde{\Psi}^{(n,\varepsilon)}_{v}(y)>
\widetilde{Z}^{(n,\varepsilon)}_{v}}-
\1_{\widetilde{\Psi}^{(n,\varepsilon)}_{v}(y)<
\widetilde{Z}^{(n,\varepsilon)}_{v}}) dv.
\end{displaymath}

\begin{lemma}
\label{LemClaim7}

For all $\varepsilon>0$, as $n\to +\infty$, the family
\begin{equation}
\label{EqDiscrBMBB}
(\phi^{(a)}(x), \widetilde{Z}^{(n,\varepsilon)}_{u},
\widetilde{\Psi}^{(n,\varepsilon)}_{u}(y),
(\widetilde{\Psi}^{(n,\varepsilon)}_{u})^{-1}(y)
)_{x\in\R, y\in\R, u\geq 0}
\end{equation}
converges in law to, for the topology of uniform convergence on compact subsets, to
\begin{displaymath}
(\phi^{(a)}(x), B_{u},
\widecheck{\Psi}_{u}(y),
(\widecheck{\Psi}_{u})^{-1}(y)
)_{x\in\R, y\in\R, u\geq 0},
\end{displaymath}
where $(B_{u})_{u\geq 0}$ is a standard Brownian motion 
starting from $0$, independent of
$\phi^{(a)}$, $(\widecheck{\Psi}_{u})_{u\geq 0}$ is the divergent Bass-Burdzy flow driven by
$(B_{u})_{u\geq 0}$, and $((\widecheck{\Psi}_{u})^{-1})_{u\geq 0}$ the inverse flow. 
\end{lemma}

\begin{proof}
For this, first we will show the tightness of the family. For the tightness of
$(\widetilde{\Psi}^{(n,\varepsilon)}_{u}(y))_{y\in\R, u\geq 0}$, we use that,
for $u\leq U^{(n)}(\widecheck{T}^{\ast (n)}_{\varepsilon})
\wedge\widecheck{\mathcal{U}}^{(n)}$,
\begin{multline*}
\widetilde{\Psi}^{(n,\varepsilon)}_{u}(y)=
(\widecheck{S}_{(U^{(n)})^{-1}(u)}^{\ast (n)}\circ
(\widecheck{S}_{0}^{\ast (n)})^{-1}(y)-\widecheck{S}_{(U^{(n)})^{-1}(u)}^{\ast (n)}(0))
\\-(\widecheck{S}_{(U^{(n)})^{-1}(u)}^{\ast (n)}(\widecheck{X}^{(n)\ast}_{(U^{(n)})^{-1}(u)})-\widecheck{S}_{(U^{(n)})^{-1}(u)}^{\ast (n)}(0))
+\widetilde{Z}^{(n,\varepsilon)}_{u},
\end{multline*}
each term having a limit in law by Lemmas
\ref{LemClaim1}, \ref{LemClaim4} and \ref{LemClaim5}, 
and that after time
$U^{(n)}(\widecheck{T}^{\ast (n)}_{\varepsilon})
\wedge\widecheck{\mathcal{U}}^{(n)}$, $\widetilde{\Psi}^{(n,\varepsilon)}_{u}$ is already a Bass-Burdzy flow. Similarly for
$((\widetilde{\Psi}^{(n,\varepsilon)}_{u})^{-1}(y))_{y\in\R, u\geq 0}$.
Further, because of the identities and bounds of 
Lemma \ref{LemClaim6}, any subsequential limit of \eqref{EqDiscrBMBB} is of form
\begin{displaymath}
(\phi^{(a)}(x), B_{u},
\bar{\Psi}_{u}(y),
(\bar{\Psi}_{u})^{-1}(y)
)_{x\in\R, y\in\R, u\geq 0},
\end{displaymath}
where $(B_{u})_{u\geq 0}$ is a standard Brownian motion 
starting from $0$, 
independent of $\phi^{(a)}$, and 
\begin{equation}
\label{EqLim1}
\bar{\Psi}_{u}(y)=\int_{0}^{u}(\1_{\bar{\Psi}_{v}(y)>B_{v}}-
\1_{\bar{\Psi}_{v}(y)<B_{v}}) dv,
\end{equation}
and thus by the uniquennes proved in \cite{BassBurdzy99StochBiff}, Theorem 2.3, 
$(\bar{\Psi}_{u})_{u\geq 0}$ is the divergent Bass Burdzy flow driven by 
$(B_{u})_{u\geq 0}$. To get
\eqref{EqLim1}, we used that
\begin{multline*}
\dfrac{\check{\lambda}^{\ast (n)}
_{t\wedge \widecheck{T}^{\ast (n)}_{\varepsilon}}
(x+2^{-n})}
{\check{\lambda}^{\ast (n)}
_{t\wedge \widecheck{T}^{\ast (n)}_{\varepsilon}}(x-2^{-n})}=\\
\dfrac
{\phi^{(a)}(x+2^{-n})^{2}
-2\ell^{\beta}_{\tau^{\beta}_{a^{2}/2}}(x+2^{-n})
+2\ell^{\beta}
_{(Q^{\Z_{n},\beta})^{-1}(Q^{\Z_{n},\beta}(\tau^{\beta}_{a^{2}/2})
-2^{n}t\wedge \widecheck{T}^{\ast (n))}}(x+2^{-n})}
{
\phi^{(a)}(x-2^{-n})^{2}
-2\ell^{\beta}_{\tau^{\beta}_{a^{2}/2}}(x-2^{-n})
+2\ell^{\beta}
_{(Q^{\Z_{n},\beta})^{-1}(Q^{\Z_{n},\beta}(\tau^{\beta}_{a^{2}/2})
-2^{n}t\wedge \widecheck{T}^{\ast (n))}}(x-2^{-n})
}
\end{multline*}
a.s. converges to $1$ as $n\to +\infty$, uniformly in $t$ and uniformly for $x$ in compact subsets of $I(\phi^{(a)})$.
\end{proof}

We are now ready to finish the proof of the Proposition \ref{PropDiscrBB}. By construction,
\begin{displaymath}
\widecheck{X}^{\ast(n)}_{(U^{(n)})^{-1}(u)
\wedge \widecheck{T}^{\ast (n)}_{\varepsilon}
\wedge t^{\ast(n)}_{\partial\Z^{\ast}_{n}}
\wedge 2^{-n}\widecheck{\mathcal{Q}}^{\Z_{n}^{\ast}}}=
(\widecheck{S}_{0}^{\ast (n)})^{-1}\circ
(\widetilde{\Psi}^{(n,\varepsilon)}_{u
\wedge U^{(n)}(\widecheck{T}^{\ast (n)}_{\varepsilon})
\wedge\widecheck{\mathcal{U}}^{(n)}})^{-1}
(\widetilde{Z}^{(n,\varepsilon)}_{u
\wedge U^{(n)}(\widecheck{T}^{\ast (n)}_{\varepsilon})
\wedge\widecheck{\mathcal{U}}^{(n)}}).
\end{displaymath}
We have that the process
$((\widecheck{S}_{0}^{\ast (n)})^{-1}\circ
(\widetilde{\Psi}^{(n,\varepsilon)}_{u})^{-1}
(\widetilde{Z}^{(n,\varepsilon)}_{u}))_{u\geq 0}$ converges in law
to the process
$((\widecheck{S}_{0}^{\ast})^{-1}\circ
(\widecheck{\Psi}_{u})^{-1}(B_{u}))_{u\geq 0}$,
which appears in Definition \ref{DefMain}, and out of which one constructs $\widecheck{X}^{\ast (n)}_{t}$ by the change of time
\begin{displaymath}
U^{\ast}(t)=\int_{0}^{t}
\check{\lambda}^{\ast}_{s}(\widecheck{X}^{\ast}_{s})^{-2} ds,
~~t\in[0,\widecheck{T}^{\ast}).
\end{displaymath}
We will also denote
\begin{displaymath}
\widecheck{T}^{\ast}_{\varepsilon}=
\sup\lbrace t\geq 0\vert \check{\lambda}^{\ast}_{t}
(\widecheck{X}^{\ast}_{t})>\varepsilon\rbrace.
\end{displaymath}
We use the fact that, as $n\to +\infty$, the joint processes
\begin{multline}
\label{EqEnumHuge}
(
\widecheck{T}^{\ast (n)}_{\varepsilon}
\wedge t^{\ast(n)}_{\partial\Z^{\ast}_{n}}
\wedge 2^{-n}\widecheck{\mathcal{Q}}^{\Z_{n}^{\ast}},
U^{(n)}(\widecheck{T}^{\ast (n)}_{\varepsilon})
\wedge\widecheck{\mathcal{U}}^{(n)},\\
\widecheck{X}^{\ast (n)}_{t\wedge
\widecheck{T}^{\ast (n)}_{\varepsilon}
\wedge t^{\ast(n)}_{\partial\Z^{\ast}_{n}}
\wedge 2^{-n}\widecheck{\mathcal{Q}}^{\Z_{n}^{\ast}}
},
\check{\lambda}^{\ast (n)}_{t\wedge
\widecheck{T}^{\ast (n)}_{\varepsilon}
\wedge t^{\ast(n)}_{\partial\Z^{\ast}_{n}}
\wedge 2^{-n}\widecheck{\mathcal{Q}}^{\Z_{n}^{\ast}}
}(x),\\
U^{(n)}(t)\wedge U^{(n)}(\widecheck{T}^{\ast (n)}_{\varepsilon})
\wedge\widecheck{\mathcal{U}}^{(n)},
(U^{(n)})^{-1}(u)\wedge\widecheck{T}^{\ast (n)}_{\varepsilon}
\wedge t^{\ast(n)}_{\partial\Z^{\ast}_{n}}
\wedge 2^{-n}\widecheck{\mathcal{Q}}^{\Z_{n}^{\ast}})
_{x\in \Z_{n}^{\ast}, t\geq 0, u\geq 0}
\end{multline}
converges a.s. to
\begin{multline*}
(
\widecheck{T}^{\beta, a}_{\varepsilon},
U(\widecheck{T}^{\beta, a}_{\varepsilon}),\\
\beta_{(\tau^{\beta}_{a^{2}/2}-t)
\wedge (\tau^{\beta}_{a^{2}/2}-\widecheck{T}^{\beta, a}_{\varepsilon})},
\phi^{(a)}(x)^{2}
-2\ell^{\beta}_{\tau^{\beta}_{a^{2}/2}}(x)
+2\ell^{\beta}
_{(\tau^{\beta}_{a^{2}/2}-t)
\wedge (\tau^{\beta}_{a^{2}/2}-\widecheck{T}^{\beta, a}_{\varepsilon})}(x),\\
U(t)\wedge U(\widecheck{T}^{\beta,a}_{\varepsilon}),
(U)^{-1}(u)\wedge\widecheck{T}^{\beta, a}_{\varepsilon})
_{x\in I(\phi^{(a)}), t\geq 0, u\geq 0}.
\end{multline*}
If we add to the family \eqref{EqEnumHuge} the processes
$(\phi^{(0)}(x)^{2},\beta_{t},
\phi^{(a)}(x)^{2})_{x\in\R, 0\leq t\leq \tau^{\beta}_{a^{2}/2}}$
and $((\widecheck{S}_{0}^{\ast (n)})^{-1}\circ
(\widetilde{\Psi}^{(n,\varepsilon)}_{u})^{-1}
(\widetilde{Z}^{(n,\varepsilon)}_{u}))_{u\geq 0}$, we get a tight family which has subsequential limits in law as $n\to +\infty$. Because of the constraints satisfied for finite $n$, any subsequential limit in law will satisfy:
\begin{itemize}
\item $\widecheck{T}^{\ast}_{\varepsilon}
=\widecheck{T}^{\beta, a}_{\varepsilon}$,
\item $\widecheck{X}^{\ast}_{t}=\beta_{(\tau^{\beta}_{a^{2}/2}-t)}$
for $t\leq \widecheck{T}^{\ast}_{\varepsilon}$,
\item $\ell^{\beta}_{\tau^{\beta}_{a^{2}/2}}(x)-
\ell^{\beta}_{\tau^{\beta}_{a^{2}/2}-t}(x)$ is the local time process
of $\widecheck{X}^{\ast}_{t}$  for 
$t\leq \widecheck{T}^{\ast}_{\varepsilon}$.
\end{itemize}
So we get the equality in law between
\begin{displaymath}
(\widecheck{X}^{\ast}_{t},\phi^{(a)}(x)^{2})
_{x\in\R,0\leq t\leq \widecheck{T}^{\ast}_{\varepsilon}}
\end{displaymath}
and
\begin{displaymath}
(\beta_{(\tau^{\beta}_{a^{2}/2}-t)},\phi^{(a)}(x)^{2})
_{x\in\R,0\leq t\leq \widecheck{T}^{\beta, a}_{\varepsilon}}.
\end{displaymath}
Taking $\varepsilon\to 0$, we get the equality in law between
\begin{displaymath}
(\widecheck{X}^{\ast}_{t},\phi^{(a)}(x)^{2})
_{x\in\R,0\leq t\leq \widecheck{T}^{\ast}}
\end{displaymath}
and
\begin{displaymath}
(\beta_{(\tau^{\beta}_{a^{2}/2}-t)},\phi^{(a)}(x)^{2})
_{x\in\R,0\leq t\leq \widecheck{T}^{\beta, a}}.
\end{displaymath}
This finishes our proof.

Note that \textit{a posteriori}, once the above identity in law established, one can show that the Brownian motion
$(B_{u})_{u\geq 0}$ driving the self repelling diffusion
$(\widecheck{X}^{\ast}_{t})_{x\in\R,0\leq t\leq \widecheck{T}^{\ast}}$
can be constructed on the same probability space as
$(\phi^{(0)}(x)^{2},\beta_{t},
\phi^{(a)}(x)^{2})_{x\in\R, 0\leq t\leq \tau^{\beta}_{a^{2}/2}}$, and
the convergence of
$(Z^{(n)}_{u\wedge U^{(n)}(\widecheck{T}^{\ast (n)}_{\varepsilon})
\wedge\widecheck{\mathcal{U}}^{(n)}})_{u\geq 0}$ to
$(B_{u\wedge U^{\ast}(\widecheck{T}^{\ast}_{\varepsilon})})_{u\geq 0}$
can be upgraded from in law as in Lemma \ref{LemClaim4} to almost sure. However, in our proof we avoid using that \textit{a priori}, and only rely on the convergence in law.
\end{proof}

Combining Theorem \ref{ThmRKInv} and 
Proposition \ref{PropElementary} (1) one immediately gets the following:

\begin{cor}
\label{CorBMRep}
Let be the triple 
\begin{displaymath}
(\phi^{(0)}(x)^{2},\beta_{t},
\phi^{(a)}(x)^{2})_{x\in\R, 0\leq t\leq \tau^{\beta}_{a^{2}/2}},
\end{displaymath}
jointly distributed as in the Ray-Knight 
coupling (Definition \ref{Def RK coupl})
and let $I(\phi^{(a)})$ be the connected component of $0$ in 
$\lbrace x\in\R\vert \phi^{(a)}(x)>0\rbrace$.
Let $I$ be another, deterministic, subinterval of $\R$
and $\check{\lambda}_{0}$ an admissible initial occupation profile on $I$. Let
$(\widecheck{X}_{t},\check{\lambda}_{t}(x))
_{x\in I, 0\leq t\leq\widecheck{T}}$ be the 
self-repelling diffusion on $I$ with initial occupation profile
$\check{\lambda}_{0}$, starting from $x_{0}\in I$.
Let
\begin{displaymath}
\widecheck{S}_{0}^{\ast}(x)=
\int_{0}^{x}\phi^{(a)}(r)^{-2} dr,
x\in I(\phi^{(a)}),
\qquad
\widecheck{S}_{0}(x)=
\int_{x_{0}}^{x}\check{\lambda}_{t}(r)^{-1} dr,
x\in I.
\end{displaymath}
Let $t\mapsto \theta(t)$ be the change of time
\begin{displaymath}
d\theta(t)=
\check{\lambda}_{0}(
\widecheck{S}_{0}^{-1}\circ\widecheck{S}_{0}^{\ast}
(\beta_{\tau^{\beta}_{a^{2}/2}-t}))^{2}
\phi^{(a)}(\beta_{\tau^{\beta}_{a^{2}/2}-t})^{-4}
dt.
\end{displaymath}
Then then process
$$(\widecheck{S}_{0}^{-1}\circ\widecheck{S}_{0}^{\ast}
(\beta_{\tau^{\beta}_{a^{2}/2}-\theta^{-1}(t)}),
((\phi^{(a)})^{2}
-2\ell^{\beta}_{\tau^{\beta}_{a^{2}/2}}
+2\ell^{\beta}_{\tau^{\beta}_{a^{2}/2}-\theta^{-1}(t)})
((\widecheck{S}_{0}^{\ast})^{-1}\circ\widecheck{S}_{0}(x)))
_{x\in I, 0\leq t\leq \theta(\widecheck{T}^{\beta}_{a})}$$
has the same law as 
$(\widecheck{X}_{t},\check{\lambda}_{t}(x))
_{x\in I, 0\leq t\leq \widecheck{T}}$.
\end{cor}

\begin{rem}
Note that the process 
$(\beta_{\tau^{\beta}_{a^{2}/2}-t})
_{0\leq t\leq \tau^{\beta}_{a^{2}/2}}$ has the same law as
$(\beta_{t})_{0\leq t\leq \tau^{\beta}_{a^{2}/2}}$, 
so the two processes can be interchanged in Theorem \ref{ThmRKInv}, 
Corollary \ref{PropLST1} and Corollary \ref{CorBMRep}.
\end{rem}

\section{Convergence for general initial occupation profile}
\label{SecBBgeneral}

In the sequel
$I$, $\widecheck{X}_{t}$, $\check{\lambda}_{t}$ 
will denote the general setting,
$\widecheck{X}^{\ast}_{t}$ and
$\check{\lambda}^{\ast}_{t}$ being reserved for the case
$\check{\lambda}^{\ast}_{0}(x)=\phi^{(a)}(x)^{2}$.
Next we show that a discrete space nearest neighbor self-repelling jump process as in Corollary \ref{PropLST1}, but with general initial occupation profile, can be embedded into a continuous self-repelling diffusion.

\begin{prop}
\label{PropEmb2}
Let $\J$ be a finite subset of $\R$ containing $0$. Let
$\check{\lambda}_{0}^{\J}$ be a positive function on $\J$. Let be
$(\widecheck{X}^{\J}_{q},\check{\lambda}_{0}^{\J}(x))
_{x\in\J, 0\leq q\leq\widecheck{\mathcal{Q}}^{\J}}$,the nearest neighbor self-repelling jump process on $\J$ introduced previously
(\eqref{EqDSRP1}, \eqref{EqDSRP2}, \eqref{EqDSRP3}), 
starting from $0$.
Let $q_{\partial\J}$ the first time $q$ when 
$\widecheck{X}^{\J}_{q}$ reaches $\min\J$ or $\max\J$.

Let $\varphi=\varphi^{\check{\lambda}_{0}^{\J}}$ be the Gaussian free field $\phi^{(a)}$, with $a=\check{\lambda}_{0}^{\J}(0)^{\frac{1}{2}}$, conditioned on $\phi^{(a)}$ being positive on
$[\min\J,\max\J]$, and on $\phi^{(a)}(x)=\check{\lambda}_{0}^{\J}(x)^{\frac{1}{2}}$ for all $x\in\J$. In other words, $\varphi/\sqrt{2}$ is obtained by interpolating between values 
$\check{\lambda}_{0}^{\J}(x)^{\frac{1}{2}}/\sqrt{2}$
for consecutive points $x\in \J$ with independent Brownian bridges conditioned on staying positive, and by adding below $\min\J$ and above
$\max\J$ two independent Brownian motions, the first one time-reversed, starting from
$\check{\lambda}_{0}^{\J}(\min\J)^{\frac{1}{2}}/\sqrt{2}$
and from
$\check{\lambda}_{0}^{\J}(\max\J)^{\frac{1}{2}}/\sqrt{2}$
respectively.

Let $I(\varphi)$ be the connected component of $0$ in the non-zero set of $\varphi$.
Let be 
$(\widecheck{X}^{\varphi}_{t},\check{\lambda}^{\varphi}_{t}(x))
_{x\in I(\varphi),0\leq t\leq \widecheck{T}^{\varphi}}$ be,
conditional on $\varphi$, the self-repelling diffusion on
$I(\varphi)$, starting from $0$, with initial occupation profile
$\check{\lambda}^{\varphi}_{0}(x)=\varphi(x)^{2}$, 
$\widecheck{T}^{\varphi}$ being the first time one of the
$\check{\lambda}^{\varphi}_{t}(x)$ reaches $0$.
Let $t_{\partial\J}^{\varphi}$ be the first time $t$ when 
$\widecheck{X}^{\varphi}_{t}$ reaches $\min\J$ or $\max\J$.

Let be
\begin{displaymath}
Q^{\J,\varphi}(t)=\sum_{x\in\J}\check{\ell}^{\varphi}_{t}(x),
\end{displaymath}
where 
$\check{\ell}^{\varphi}_{t}(x)=
(\check{\lambda}^{\varphi}_{0}(x)-\check{\lambda}^{\varphi}_{t}(x))/2$
is the local time process of $\widecheck{X}^{\varphi}_{t}$. 
Denote $(Q^{\J,\varphi})^{-1}$ the
right-continuous inverse of $Q^{\J,\varphi}$. Then the process
\begin{equation}
\label{EqP1}
(\widecheck{X}^{\varphi}_{(Q^{\J,\varphi})^{-1}(q)},
\check{\lambda}^{\varphi}_{(Q^{\J,\varphi})^{-1}(q)}(x))
_{x\in\J, 0\leq q\leq 
Q^{\J,\varphi}(\widecheck{T}^{\varphi})\wedge
Q^{\J,\varphi}(t_{\partial\J}^{\varphi})}
\end{equation}
has the same law as
\begin{equation}
\label{EqP2}
(\widecheck{X}^{\J}_{q},\check{\lambda}_{0}^{\J}(x))
_{x\in\J, 0\leq q\leq\widecheck{\mathcal{Q}}^{\J}
\wedge q_{\partial\J}}.
\end{equation}
\end{prop}

\begin{proof}
For $(\check{\lambda}^{\J}_{0}(x))_{x\in\J\setminus\lbrace 0\rbrace}$
not fixed, but random, distributed as 
$(\phi^{(a)}(x)^{2})_{x\in\J\setminus\lbrace 0\rbrace}$,
$\phi^{(a)}$ being conditioned on being positive on $[\min\J,\max\J]$,
the identity in law is a direct consequence of
Corollary \ref{PropLST1} and 
Theorem \ref{ThmRKInv}. To conclude that the identity in law disintegrated according the values of 
$(\check{\lambda}^{\J}_{0}(x))_{x\in\J\setminus\lbrace 0\rbrace}$ also holds, it is sufficient to show that both sides of the identity,
\eqref{EqP1} and \eqref{EqP2}, are continuous with respect to
$(\check{\lambda}^{\J}_{0}(x))_{x\in\J\setminus\lbrace 0\rbrace}$. The continuity of the law of 
\eqref{EqP2} with respect to 
$(\check{\lambda}^{\J}_{0}(x))_{x\in\J\setminus\lbrace 0\rbrace}$
is clear from the construction.
As for \eqref{EqP1}, first the law of
$(\varphi(x))_{x\in[\min\J,\max\J]}$, hence the law of
$(\check{\lambda}^{\varphi}_{0}(x))_{x\in[\min\J,\max\J]}$, 
depends continuously on
$(\check{\lambda}^{\J}_{0}(x))_{x\in\J\setminus\lbrace 0\rbrace}$, 
and second, according to Lemma \ref{LemContinuity}, the law of
\eqref{EqP1} depends continuously on 
$(\check{\lambda}^{\varphi}_{0}(x))_{x\in[\min\J,\max\J]}$.
\end{proof}

\begin{proof}[Proof of Theorem \ref{ThmSIConv}]
We will first consider the case of $I$ bounded.
Without loss of generality, we assume that $0\in I$ and
$\widecheck{X}_{0}=0$. 
We also slightly simplify by taking 
$\widecheck{X}_{0}^{(n)}=\widecheck{X}_{0}=0$
for all $n$.
Using the notations of 
Proposition \ref{PropEmb2}, let be
$\J^{(n)}=2^{-n}\Z\cap I$ and
$\varphi^{(n)}$ the conditioned GFF interpolating between
$(\check{\lambda}_{0}(x)^{\frac{1}{2}})_{x\in\J^{(n)}}$.
By Proposition \ref{PropEmb2}, we can take
\begin{equation}
\label{EqInterpol}
\widecheck{X}^{(n)}_{t}=
\widecheck{X}^{\varphi^{(n)}}_{
(Q^{\J^{(n)},\varphi^{(n)}})^{-1}(2^{n} t)
},
\check{\lambda}^{(n)}_{t}(x)=
\widecheck{X}^{\varphi^{(n)}}_{
(Q^{\J^{(n)},\varphi^{(n)}})^{-1}(2^{n} t)
}(x),
\end{equation}
$$
t\leq 2^{-n}Q^{\J^{(n)},\varphi^{(n)}}
(\widecheck{T}^{\varphi^{(n)}}\wedge 
t^{\varphi^{(n)}}_{\partial \J^{(n)}}),
$$
where $t^{\varphi^{(n)}}_{\partial \J^{(n)}}$
is the first time 
$\widecheck{X}^{\varphi^{(n)}}_{t}$ hits
$\min \J^{(n)}$ or $\max\J^{(n)}$.

\medskip

\begin{lemma}
\label{LemC1}

As $n\to +\infty$, $(\varphi^{(n)}(x))_{x\in I}$ converges in probability to 
$(\check{\lambda}_{0}(x)^{\frac{1}{2}})_{x\in I}$
for the topology of uniform converge on compact subsets of $I$.
\end{lemma}

\begin{proof}
Indeed, given $K$ a compact subinterval of $I$ and $n$ large enough so that $K\subseteq [\min\J^{(n)},\max\J^{(n)}]$, one will obtain 
$\varphi^{(n)}$ by first interpolating linearly between the values
of $(\check{\lambda}_{0}(x)^{\frac{1}{2}})_{x\in\J^{(n)}}$, then by adding of order $2^{n}$ independent bridges from $0$ to $0$ of
duration $2^{-n}$, each conditioned by a positivity event. The minimal probability of an event by which we condition will converge to $1$ with $n$. Moreover, for an unconditioned bridge, the probability
to deviate more than $\varepsilon$ from $0$ is
$O(\exp(-k 2^{n} \varepsilon^{2}))$, for a constant $k>0$. 
This beats the $2^{n}$ factor.
\end{proof}

\begin{lemma}
\label{LemC2}

As $n\to +\infty$, the process
$(\widecheck{X}^{\varphi^{(n)}}
_{t\wedge \widecheck{T}^{\varphi^{(n)}}
\wedge t^{\varphi^{(n)}}_{\partial \J^{(n)}}},
\check{\lambda}^{\varphi^{(n)}}
_{t\wedge \widecheck{T}^{\varphi^{(n)}}
\wedge t^{\varphi^{(n)}}_{\partial \J^{(n)}}}
(x))_{x\in I, t\geq 0}$
converges in law to
$(\widecheck{X}_{t\wedge \widecheck{T}},
\check{\lambda}_{t\wedge \widecheck{T}}
(x))_{x\in I, t\geq 0}$.
\end{lemma}

\begin{proof}
Indeed, by Lemma \ref{LemC1}, 
$(\check{\lambda}^{\varphi^{(n)}}_{0}(x))_{x\in I}$
converges in probability to
$(\check{\lambda}_{0}(x))_{x\in I}$
for the topology of uniform convergence on compact subsets,
the law of the self-repelling diffusion depends continuously on the initial occupation profile
(Lemma \ref{LemContinuity}), and the range of
$(\widecheck{X}_{t\wedge \widecheck{T}})_{t\geq 0}$
is a.s. a compact subinterval of $I$.
\end{proof}

\begin{lemma}
\label{LemC2Bis}

As $n\to +\infty$, simultaneously with the convergence in law of Lemma \ref{LemC2}, we have that
$t\mapsto 2^{-n}Q^{\J^{(n)},\varphi^{(n)}}
(t\wedge \widecheck{T}^{\varphi^{(n)}}
\wedge t^{\varphi^{(n)}}_{\partial \J^{(n)}})
$ converges in 
law to $t\mapsto t\wedge\widecheck{T}$  for the uniform topology.
\end{lemma}

\begin{proof}
To simplify, we will assume here that all the
$$(\widecheck{X}^{\varphi^{(n)}}
_{t\wedge \widecheck{T}^{\varphi^{(n)}}
\wedge t^{\varphi^{(n)}}_{\partial \J^{(n)}}},
\check{\lambda}^{\varphi^{(n)}}
_{t\wedge \widecheck{T}^{\varphi^{(n)}}
\wedge t^{\varphi^{(n)}}_{\partial \J^{(n)}}}
(x))_{x\in I, t\geq 0}$$ and  
$(\widecheck{X}_{t\wedge \widecheck{T}},
\check{\lambda}_{t\wedge \widecheck{T}}
(x))_{x\in I, t\geq 0}$ live on the same probability space, constructed from the same driving Brownian motion $(B_{u})_{u\geq 0}$, independent of the $\varphi^{(n)}$. This is always possible to do.
Write
\begin{eqnarray*}
2^{-n}Q^{\J^{(n)},\varphi^{(n)}}(t\wedge 
\widecheck{T}^{\varphi^{(n)}})&=&2^{-n-1}\sum_{x\in\J^{(n)}}
(\check{\lambda}^{\varphi^{(n)}}_{0}(x)-
\check{\lambda}^{\varphi^{(n)}}
_{t\wedge \widecheck{T}^{\varphi^{(n)}}
}(x))
\\&=&
2^{-n-1}\sum_{x\in\J^{(n)}}
(\check{\lambda}^{\varphi^{(n)}}_{0}(x)-
\check{\lambda}_{0}(x)-
\check{\lambda}^{\varphi^{(n)}}
_{t\wedge \widecheck{T}^{\varphi^{(n)}}}(x)
+\check{\lambda}_{t\wedge\widecheck{T}}(x))\\
&&+2^{-n-1}\sum_{x\in\J^{(n)}}
(\check{\lambda}_{0}(x)-\check{\lambda}_{t\wedge\widecheck{T}}(x)).
\end{eqnarray*}
We have that
$$2^{-n-1}\sum_{x\in\J^{(n)}}
(\check{\lambda}_{0}(x)-\check{\lambda}_{t\wedge\widecheck{T}}(x))$$
converges a.s. to $t\wedge \widecheck{T}$, uniformly of 
$[0,+\infty)$. Moreover,
\begin{multline*}
\big\vert
2^{-n-1}\sum_{x\in\J^{(n)}}
(\check{\lambda}^{\varphi^{(n)}}_{0}(x)-
\check{\lambda}_{0}(x)-
\check{\lambda}^{\varphi^{(n)}}
_{t\wedge \widecheck{T}^{\varphi^{(n)}}}(x)
+\check{\lambda}_{t\wedge\widecheck{T}}(x))
\big\vert
\\
\leq
(1+\vert I\vert)
\dfrac{1}{2}\max_{x\in\J^{(n)},s\geq 0}
\vert
\check{\lambda}^{\varphi^{(n)}}_{0}(x)-
\check{\lambda}_{0}(x)-
\check{\lambda}^{\varphi^{(n)}}
_{s\wedge \widecheck{T}^{\varphi^{(n)}}}(x)
+\check{\lambda}_{s\wedge\widecheck{T}}(x)
\vert,
\end{multline*}
$\vert I\vert$ being the length of $I$,
and the right-hand side converges in probability to $0$.
Finally,
$t^{\varphi^{(n)}}_{\partial \J^{(n)}}>
\widecheck{T}^{\varphi^{(n)}}$
with probability converging to $1$.
\end{proof}

\begin{lemma}
\label{LemC3} 

As $n\to +\infty$, the process
$$(\widecheck{X}^{(n)}
_{t\wedge 
2^{-n}Q^{\J^{(n)},\varphi^{(n)}}
(\widecheck{T}^{\varphi^{(n)}}
\wedge t^{\varphi^{(n)}}_{\partial \J^{(n)}})
},
\check{\lambda}^{(n)}
_{t\wedge 
2^{-n}Q^{\J^{(n)},\varphi^{(n)}}
(\widecheck{T}^{\varphi^{(n)}}
\wedge t^{\varphi^{(n)}}_{\partial \J^{(n)}})
}
(x))_{x\in \J^{(n)}, t\geq 0}$$
converges in law to
$(\widecheck{X}_{t\wedge \widecheck{T}},
\check{\lambda}_{t\wedge \widecheck{T}}
(x))_{x\in I, t\geq 0}$.
\end{lemma}

\begin{proof}
This follows from 
\eqref{EqInterpol}, Lemma \ref{LemC2} and the convergence of 
$$t\mapsto 2^{-n}Q^{\J^{(n)},\varphi^{(n)}}
(t\wedge \widecheck{T}^{\varphi^{(n)}}
\wedge t^{\varphi^{(n)}}_{\partial \J^{(n)}})
$$ in 
law to $t\mapsto t\wedge\widecheck{T}$ (Lemma \ref{LemC2Bis}).
\end{proof}

To finish the proof of Theorem \ref{ThmSIConv}, 
observe that by Lemma \ref{LemC3}, 
$$\check{\lambda}^{(n)}_{
2^{-n}Q^{\J^{(n)},\varphi^{(n)}}
(\widecheck{T}^{\varphi^{(n)}}
\wedge t^{\varphi^{(n)}}_{\partial \J^{(n)}})
}
(\widecheck{X}^{(n)}_{
2^{-n}Q^{\J^{(n)},\varphi^{(n)}}
(\widecheck{T}^{\varphi^{(n)}}
\wedge t^{\varphi^{(n)}}_{\partial \J^{(n)}})
})$$
converges in probability to
$\check{\lambda}_{\widecheck{T}}(\widecheck{X}_{\widecheck{T}})=0$,
thus $\widecheck{T}^{(n)}_{\varepsilon}<
2^{-n}Q^{\J^{(n)},\varphi^{(n)}}
(\widecheck{T}^{\varphi^{(n)}}
\wedge t^{\varphi^{(n)}}_{\partial \J^{(n)}})$ with probability converging to 1.

Finally, if $I$ is unbounded, it is enough to consider an increasing family of bounded subintervals of $I$ which at the limit gives $I$, as the range of
$\widecheck{X}_{t\wedge \widecheck{T}_{\varepsilon}}$ is a.s. bounded.
\end{proof}

\section*{Acknowledgements}

This work was supported by the French National Research Agency (ANR) grant
within the project MALIN (ANR-16-CE93-0003).

This work was partly supported by the LABEX MILYON (ANR-10-LABX-0070) of Université de Lyon, within the program "Investissements d'Avenir" (ANR-11-IDEX-0007) operated by the French National Research Agency (ANR).

TL acknowledges the support of Dr. Max Rössler, the Walter Haefner
Foundation and the ETH Zurich Foundation.

PT acknowledges the support of the National Science Foundation of China (NSFC), grant No. 11771293.

\bibliographystyle{alpha}
\bibliography{titusbibnew}

\end{document}